\documentclass{amsart}

\usepackage{mathrsfs}
\usepackage{amsmath}
\usepackage{amsthm}

\usepackage{mathtools}
\usepackage{amssymb}
\usepackage{amscd}
\usepackage{appendix}
\usepackage{amsfonts}
\usepackage[all]{xy}
\usepackage{array}
\usepackage{graphicx}
\usepackage{longtable}
\usepackage[OT2,T1]{fontenc}
\usepackage[utf8]{inputenc}
\usepackage{tikz}
\usepackage{makecell}
\DeclareSymbolFont{cyrletters}{OT2}{wncyr}{m}{n}
\DeclareMathSymbol{\Sha}{\mathalpha}{cyrletters}{"58}

\usepackage[hypertexnames=false]{hyperref}

\theoremstyle{plain}
\newtheorem{thm}{Theorem}[section]
\newtheorem{lem}[thm]{Lemma}
\newtheorem{prop}[thm]{Proposition}
\newtheorem{cor}[thm]{Corollary}

\theoremstyle{definition}
\newtheorem{defn}[thm]{Definition}

\theoremstyle{remark}
\newtheorem{rem}[thm]{Remark}

\newcommand{\FF}{{\mathbb F}}

\newcommand{\ZZ}{{\mathbb Z}}
\newcommand{\QQ}{{\mathbb Q}}

\newcommand{\CC}{{\mathbb C}}

\newcommand{\PP}{{\mathbb P}}
\newcommand{\TT}{{\mathbb T}}

\newcommand{\afk}{\mathfrak{a}}

\newcommand{\dfk}{\mathfrak{d}}
\newcommand{\mfk}{\mathfrak{m}}
\newcommand{\nfk}{\mathfrak{n}}
\newcommand{\pfk}{\mathfrak{p}}
\newcommand{\qfk}{\mathfrak{q}}

\newcommand{\Gcal}{\mathcal{G}}
\newcommand{\Hcal}{\mathcal{H}}

\newcommand{\Scal}{\mathcal{S}}

\newcommand{\Ical}{\mathcal{I}}
\newcommand{\Ccal}{\mathcal{C}}

\newcommand{\sT}{\mathscr{T}}

\newcommand{\ord}{\operatorname{ord}}

\newcommand{\GL}{\operatorname{GL}}
\newcommand{\SL}{\operatorname{SL}}

\newcommand{\PGL}{\operatorname{PGL}}

\DeclareMathOperator{\End}{End}

\newcommand{\Stab}{\operatorname{Stab}}

\newcommand\subfrac[2]{\genfrac{}{}{0pt}{}{#1}{#2}}

\numberwithin{equation}{section}

\usepackage{xcolor}

\begin{document}
\title{Sturm-type bounds for modular forms over function fields}

\author{Cécile Armana}
\address{Laboratoire de Mathématiques de Besançon, Université Bourgogne Franche-Comté, CNRS UMR-6623. 16, route de Gray
25030 Besançon Cedex, France.}
\email{cecile.armana@univ-fcomte.fr}

\author{Fu-Tsun Wei}
\address{Department of Mathematics, National Tsing Hua University, No.\ 101, Section 2, Kuang-Fu Road, Hsinchu City 30013, Taiwan}
\email{ftwei@math.nthu.edu.tw}


\begin{abstract}
In this paper, we obtain two analogues of the Sturm bound for modular forms in the function field setting. In the case of mixed characteristic, we prove that any harmonic cochain is uniquely determined by an explicit finite number of its first Fourier coefficients where our bound is much smaller than the ones in the literature. A similar bound is derived for generators of the Hecke algebra on harmonic cochains. As an application, we present a computational criterion for checking whether two elliptic curves over the rational function field $\FF_q(\theta)$ with same conductor are isogenous. In the case of equal characteristic, we also prove that any Drinfeld modular form is uniquely determined by an explicit finite number of its first coefficients in the $t$-expansion.
\end{abstract}

\maketitle

\section*{Introduction}

The Sturm bound provides a sufficient condition for classical modular forms to be identically zero.

\begin{thm}\label{thm: classical}
\text{\rm (Sturm \cite{S}, see also \cite[Cor.~2.3.4]{M} and \cite[Chap.~9]{Stein})} Let $k$ and $N$ be positive integers. Given a modular form $f$ of weight $k$ for the congruence subgroup $\Gamma_0(N)$, consider its Fourier expansion $f(z) = \sum_{n=0}^\infty c_n(f)e^{2\pi i n z}$. Then $f$ is identically zero if $$c_n(f) = 0 \quad \text{ for } 0\leq n \leq [\SL_2(\ZZ): \Gamma_0(N)] \cdot \frac{k}{12}.$$
Moreover, if $f$ is a cusp form, then $f$ is identically zero if
$$c_n(f) = 0 \quad \text{ for }  0\leq n \leq [\SL_2(\ZZ):\Gamma_0(N)]\cdot \left(\frac{k}{12} - \frac{1}{N}\right) + \frac{1}{N}.$$
\end{thm}

Let $m$ denote the dimension of the $\CC$-vector space $M_k(\Gamma_0(N))$ of weight-$k$ modular forms for $\Gamma_0(N)$. Comparing the former bound with $m$, the theorem says that a given modular form in $M_k(\Gamma_0(N))$ is uniquely  determined by only slightly more than its first $m$ Fourier coefficients. Such bounds have many theoretical and computational applications, in particular they are widely used in algorithms for computing with modular forms. 

Let $\TT_k(N)$ be the Hecke algebra acting on the space $S_k(\Gamma_0(N))$ of weight-$k$ cusp forms for $\Gamma_0(N)$. It is well-known that the first Fourier coefficient provides a perfect pairing between $\TT_k(N)$ and $S_k(\Gamma_0(N))$. As a consequence, one can derive from Theorem~\ref{thm: classical} an explicit bound for the number of Hecke operators generating  $\TT_k(N)$ (cf.~\cite{Stein}). Moreover, combined with the modularity theorem for elliptic curves over $\QQ$, the Sturm bound can be used to check efficiently whether two elliptic curves over $\QQ$ are isogenous. Note also that the statement of Theorem~\ref{thm: classical} is the version at the \lq\lq generic prime\rq\rq: the bound actually holds as well at every \lq\lq closed prime\rq\rq\ for arithmetic modular forms and is essential in the study of their congruence relations. 

Recently, Sturm-type bounds for Hilbert modular forms and Siegel modular forms have been the subject of several investigations. The aim of this paper is to give an attempt on studying the generic version of this question for modular forms over function fields, in both cases of mixed and equal characteristic.

\subsection{Mixed characteristic setting}

Let $K := \FF_q(\theta)$ be the rational function field with one variable $\theta$ over a finite field $\FF_q$ with $q$ elements.
Let $K_\infty$ be the completion of $K$ with respect to the infinite place $\infty$. Put $A := \FF_q[\theta]$ and denote by $A_+$ the set of monic polynomials in $A$.

A combinatorial analogue of the complex upper half plane in this setting is the Bruhat-Tits tree $\sT$ associated to $\PGL_2(K_\infty)$.
We are interested in \emph{harmonic cochains}, also called Drinfeld-type automorphic forms, which are functions on the set of the oriented edges of $\sT$ satisfying the so-called \emph{harmonicity property.}
Harmonic cochains, which can be viewed as analogue to classical weight-$2$ modular forms, are objects of great interest in the study of function field arithmetic (for instance cf. \cite{G-R}, \cite{R-T}, \cite{Wei}, \cite{W-Y}).
Moreover let $\Gamma_0(\nfk)$ be the Hecke congruence subgroup of $\GL_2(A)$ for a given $\nfk \in A_+$. The space of $\Gamma_0(\nfk)$-invariant $\CC$-valued harmonic cochains is denoted by $\Hcal(\nfk)$. Every $f$ in $\Hcal(\nfk)$ admits a unique Fourier expansion with coefficients $c_0(f)$ and $(c_\mfk(f))_{\mfk \in A_+}$.

Classical Sturm bounds may be proved using the so-called valence formula for modular forms. Since no such formula is available for harmonic cochains, a more natural approach is to use a fundamental domain of the quotient graph $\Gamma_0(\nfk) \backslash \sT$, as we describe in Theorem~\ref{thm: Quo-G-dec} for instance. Our first bound is stated in terms of the arithmetic quantity $\tau(\nfk)$ introduced in \text{\rm Definition~\ref{defn: d(n)}}.

\begin{thm}\label{thm: 0.3}
Given $\nfk \in A_+$, let $f \in \Hcal(\nfk)$. Then $f$ is identically zero if $c_\mfk(f) = 0$ for all $\mfk \in A_+$ with
$$
\deg \mfk \leq \deg \nfk -1 + 2 \tau(\nfk).
$$
\end{thm}
One may observe that $\tau(\nfk) \leq \deg \nfk -2$ when $\deg \nfk \geq 2$. If $t(\nfk)$ denotes the number of prime factors of $\nfk$, we also have the following special values:
$$
\tau(\nfk) = 
\begin{cases}
0 & \text{ if $0 \leq t(\nfk) \leq q$,} \\
1 & \text{ if $q<t(\nfk)\leq 2q$.}
\end{cases}
$$

\begin{rem}
It seems difficult to tell the precise value of $\tau(\nfk)$ when $t(\nfk) > 2q$.
However, from the numerical data in \eqref{V2} and \eqref{Value of t(m,n) q=3} which were computed using SageMath, we predict that 
\begin{equation}\label{eq: predicttau}
\tau(\nfk)\ ?\!\!\leq \left\lfloor \frac{t(\nfk)-1}{q}\right\rfloor.
\end{equation}
Moreover, when $q$ and $\deg \nfk$ are small, data \eqref{compare bounds} shows that our Sturm-type bound is actually sharp in certain cases.
\end{rem}

\subsubsection{Cuspidal harmonic cochains} 

Let $\Hcal_0(\nfk)$ denote the subspace of cuspidal $\Gamma_0(\nfk)$-invariant $\CC$-valued harmonic cochains, which consists of elements of $\Hcal(\nfk)$ which are finitely supported modulo $\nfk$. The next bound is given in terms of the arithmetic quantity $\ell(\nfk)$ introduced in \text{\rm Definition~\ref{defn: l(n)}}.

\begin{thm}\label{thm: 0.4}
Given $\nfk \in A_+$, let $f \in \Hcal_0(\nfk)$. Then $f$ is identically zero if $c_\mfk(f) = 0$ for all $\mfk \in A_+$ with
\begin{eqnarray}
\deg \mfk &\leq& \deg \nfk - 2 \nonumber \\
&& +
\begin{cases} 0 & \text{ if $\nfk$ is a prime power,}\\
0 & \text{ if $\nfk$ is square-free and $f$ is \lq\lq new\rq\rq,}\\
0 & \text{ if $\nfk = \pfk^2\qfk$ for primes $\pfk,\qfk \in A_+$ with $\deg \qfk =1$ and $f$ is \lq\lq new\rq\rq,} \\
\ell(\nfk) & \text{ otherwise.}
\end{cases} \nonumber 
\end{eqnarray}
\end{thm}

\begin{rem}\label{rem: smalldeg}
When $\deg \nfk <3$, it is known that $\Hcal_0(\nfk) =\{ 0\}$ by the genus formula for $\Gcal(\nfk)$ (\cite[Th.~2.17]{G-N}). Thus when $\deg \nfk = 3$, every $f \in \Hcal_0(\nfk)$ is \lq\lq new\rq\rq\ and Theorem~\ref{thm: 0.4} says that $f$ is identically zero if $c_\mfk(f) = 0$ for all $\mfk \in A_+$ with $\deg \mfk \leq 1$.
\end{rem}

We point out that $\ell(\nfk)$ and $\tau(\nfk)$ are defined in very different ways. From a computational point of view, it is relatively harder to determine the value $\ell(\nfk)$ than $\tau(\nfk)$.
However, we can show that $\ell(\nfk) \leq 2 \tau(\nfk)+1$ for every $\nfk \in A_+$ (Corollary \ref{cor: SBound.2}), which indicates that 
the bound in Theorem~\ref{thm: 0.4} is better than the one in Theorem~\ref{thm: 0.3}.

\subsubsection{Hecke algebra on harmonic cochains}

Similarly to the classical case, the pairing between the Hecke algebra and the space of $\CC$-valued harmonic cochains coming from the first Fourier coefficient $c_1$ is indeed perfect, cf.\ Lemma~\ref{lem: perfect-pairing}, and the action of the Hecke algebra can be seen actually from the Fourier expansion. Consequently, the previous bounds allow an explicit control on the number of Hecke operators which generate the Hecke algebra.

\begin{cor}\label{cor: 0.5}
${}$
\begin{enumerate}
\item Let $\TT(\nfk)$ be the Hecke algebra acting on $\Hcal(\nfk)$. Then $\TT(\nfk)$ is spanned as a $\CC$-vector space by $T_\mfk$ for all $\mfk \in A_+$ with
$$
\deg \mfk \leq \deg \nfk -1 + 2 \tau(\nfk).
$$
\item Let\ $\TT_0(\nfk)$
be the Hecke algebra acting on $\Hcal_0(\nfk)$.
Then $\TT_0(\nfk)$ is spanned as a $\CC$-vector space by $T_\mfk$ for all $\mfk \in A_+$ with
$$
\deg \mfk \leq 
\deg \nfk - 2 +
\begin{cases} 0 & \text{ if $\nfk$ is a prime power,}\\
\ell(\nfk) & \text{ otherwise.}
\end{cases}
$$
\item Let $\TT_0^{\text{\rm new}}(\nfk)$ be the restriction of $\TT_0(\nfk)$ acting on the \lq\lq new\rq\rq\ subspace of $\Hcal_0(\nfk)$. 
If $\nfk$ is either square-free or $\nfk = \pfk^2 \qfk$ for primes $\pfk,\qfk \in A_+$ with $\deg \qfk =1$, then
$\TT_0^{\text{\rm new}}(\nfk)$ is spanned as a $\CC$-vector space by $T_\mfk$ for all $\mfk \in A_+$ with $\deg \mfk \leq \deg \nfk -2$.
\end{enumerate}
\end{cor}


\begin{rem}
Using Theorem~\ref{thm: Quo-G-dec}, we also get the coarse bounds $2\deg \nfk -4$ for $\Hcal_0(\nfk)$ and $\TT_0(\nfk)$, and $\max (2\deg \nfk-3,\deg \nfk -1)$ for $\Hcal(\nfk)$ and $\TT(\nfk)$ (Proposition~\ref{prop: SBound-0}, Remarks \ref{rem: otherboundsforH}\eqref{rem: otherboundsforHcoarse} and \ref{rem: coarseboundforHecke}). The bounds of Theorems \ref{thm: 0.3}, \ref{thm: 0.4} and Corollary \ref{cor: 0.5} are obtained by twisting the fundamental domain by the Atkin-Lehner involution $W_\nfk$; computational data, in particular \eqref{eq: predicttau}, suggest that they are smaller than the coarse bounds.
\end{rem}

\subsubsection{Review of previous Sturm-type bounds for harmonic cochains}

Tan and Rockmore \cite{T-R} proved a Sturm bound for certain general automorphic cusp forms on $\GL_2$ over $K$: for harmonic cochains, one can derive bounds of the form $5 \deg \nfk + 5$ for normalized Hecke eigenforms (\cite[Section 3, p.~128]{T-R}), and $\deg \nfk - 2$ under the further assumption that $\nfk$ is squarefree (\cite[Section 4, p.~131]{T-R}) (note that their level $N$ corresponds here to $\nfk \, \infty$). It can be compared with the square-free case of Theorem~\ref{thm: 0.4} where we only assumed that the harmonic cochain is \lq\lq new\rq\rq.

When $\deg \nfk=3$, it is known since Gekeler \cite[5.8 and 7.1]{Gek1} that any $f \in \Hcal_0(\nfk)$ is identically zero when $c_\mfk(f)=0$ for all $\mfk\in A_+$ with $\deg \mfk \leq 1$. The same bound can be derived for the corresponding cuspidal Hecke algebra (\cite[Theorem 1.4 (iii)]{P-W}). These results are recovered by Theorem~\ref{thm: 0.4} and Corollary~\ref{cor: 0.5} (see Remark~\ref{rem: smalldeg}).

In order to improve on existing bounds, our input is to carefully describe the quotient graph $\Gamma_0(\nfk) \backslash \sT$ and to utilize the harmonicity property. We mention that although Gekeler and Nonnengardt \cite{G-N} have worked on the structure of this graph, no Sturm bound seems to appear explicitly in their paper, although it is possible that some bound can be derived.

\subsubsection{Isogeny between elliptic curves}

 Let $E$ be an elliptic curve over $K$ with split multiplicative reduction at the place $\infty$. Denote by $\nfk \, \infty$ the conductor of $E$ with $\nfk \in A_+$. From the work of Weil, Jacquet-Langlands, Grothendieck, Deligne, Drinfeld and Zarhin, there exists a unique $\Gamma_0(\nfk)$-invariant $\CC$-valued cuspidal \lq\lq 
  new\rq\rq\ harmonic cochain $f_E$ corresponding to the $K$-isogeny class of $E$ (\cite{G-R}). Combined with Theorem~\ref{thm: 0.4} applied to $f_E$, we get the following isogeny criterion.

\begin{cor}\label{cor: 0.6}
Let $E_1$ and $E_2$ be two elliptic curves over $K$ with the same conductor $\nfk \, \infty$ and split multiplicative reduction at $\infty$. Then $E_1$ and $E_2$ are isogenous over $K$ if and only if $a_\pfk(E_1) = a_\pfk(E_2)$ for every prime $\pfk \in A_+$ with
$$
\deg \pfk \leq 
\deg \nfk -2 +
\begin{cases}
0 & \text{ if $\nfk$ is a prime power,} \\
0 & \text{ if $\nfk$ is square-free,}\\
0 & \text{ if $\nfk = \qfk_1^2\qfk_2$ for primes $\qfk_1,\qfk_2 \in A_+$ with $\deg \qfk_2 =1$,} \\
\ell(\nfk) & \text{ otherwise.}
\end{cases}
$$
Here $a_\pfk(E)$ is introduced in \eqref{eqn: ap(E)}.
\end{cor}

\subsection{Equal characteristic setting}

Let $\CC_\infty$ be the completion of a chosen algebraic closure of $K_\infty$. 
Set $\Omega:= \CC_\infty - K_\infty$, the Drinfeld half plane. Let $k,m$ be non-negative integers with $0\leq m\leq q-2$. Given $\nfk \in A_+$, recall that a Drinfeld modular form $f$ of weight $k$ and type $m$ for the congruence subgroup $\Gamma_0(\nfk)$ admits a so-called \emph{$t$-expansion}:
$$f = \sum_{j=0}^\infty b_j(f)t^{m+(q-1)j},$$
where $t:\Omega\rightarrow \CC_\infty$ is a chosen uniformizer at the cusp infinity. We obtain a Sturm-type bound for Drinfeld modular forms which generalizes Gekeler \cite[Corollary~5.17]{Gek5} in the case $\nfk=1$.

\begin{thm}\label{thm: 0.2}
Given $\nfk \in A_+$, let $f$ be an $\ell$-cuspidal Drinfeld modular form of weight $k$ and type $m$ for $\Gamma_0(\nfk)$ as defined in Section~\ref{sec: SBoundD}. Then $f$ is identically zero if
$$b_j(f) = 0 \quad \text{ for all }\ 0\leq j \leq [\GL_2(A):\Gamma_0(\nfk)] \cdot \left(\frac{k}{q^2-1}-\frac{\ell}{(q-1)q^{\deg \nfk}}\right)+ \frac{\ell - m q^{\deg \nfk}}{(q-1)q^{\deg \nfk}}.$$
\end{thm}

This bound is essentially similar to the classical Sturm bound  (Theorem~\ref{thm: classical}) and is proved likewise.
However, given $\nfk \in A_+$, $k, \ell \in \ZZ_{\geq 0}$, and an integer $m$ with $0\leq m \leq q-2$, the pairing between the space of $\ell$-cuspidal Drinfeld modular forms of weight $k$ and type $m$ for $\Gamma_0(\nfk)$ and the associated Hecke algebra, given by the first coefficient $b_1$, is not expected to be perfect (cf. Section~\ref{sec: SBoundD2}). Besides it is not obvious to how read off the action of the Hecke algebra on Drinfeld modular forms via their $t$-expansions.
Thus, differently from the cases of classical modular forms and harmonic cochains, the bound of Theorem~\ref{thm: 0.2} does not give directly a bound for generators of the Hecke algebra on Drinfeld modular forms.

\subsection{Content}

This paper is organized as follows. We set up basic notations in Section~\ref{Sec: Pre}. In Section~\ref{sec: Quo-G}, we review the structure of the quotient graph of $\sT$ by congruence subgroups $\Gamma_0(\nfk)$ and in Section~\ref{Sec-Har} the needed properties of harmonic cochains. In Section~\ref{Sec: Sbound}, we first prove Theorem~\ref{thm: 0.4} for cuspidal harmonic cochains in Section~\ref{subsec: Cusp}; Theorem~\ref{thm: 0.3} for harmonic cochains is obtained in Section~\ref{subsec: Non-cusp}.
Section~\ref{sec: App} includes applications of our Sturm-type bounds for harmonic cochains to the Hecke algebra and to isogenies between elliptic curves: Corollary~\ref{cor: 0.5} is shown in Section~\ref{subsec: Hec-Alg} and Section~\ref{subsec: new-space}, and Corollary~\ref{cor: 0.6} is derived in Section~\ref{subsec: Ell}.
Finally, we prove Theorem~\ref{thm: 0.2} for Drinfeld modular forms in Section~\ref{sec: SBoundD2}.

\section{Preliminaries}\label{Sec: Pre}

\subsection{Notations}\label{subsec: Not}

Let $\FF_q$ be a finite field with $q$ elements and $K := \FF_q(\theta)$, the rational function field with one variable $\theta$ over $\FF_q$. Let $A:= \FF_q[\theta]$ be the ring of integers of $K$ and $A_+$ be the set of monic polynomials in $A$. 
The degree valuation on $K$, i.e.\ the valuation corresponding to the infinite place $\infty$ of $K$, is defined by:
    $$\forall a,b \in A \text{ with } b \neq 0,\quad\nu_\infty(a/b) := \deg b - \deg a$$
and the corresponding absolute value is normalized to be:
$$\forall \alpha \in K,\quad \left|\alpha \right|_\infty := q^{-\nu_\infty(\alpha)}. $$
Take $\pi_\infty := \theta^{-1}$, a uniformizer at $\infty$.
Let $K_\infty :=\FF_q (\!(\pi_\infty)\!)$ be the completion of $K$ with respect to $|\cdot|_\infty$,
and set $O_\infty := \FF_q [\![\pi_\infty]\!]$, 
the ring of integers in $K_\infty$.

\subsection{Bruhat-Tits Tree}\label{subsec: Tree}

Let $\sT$ be the Bruhat-Tits tree associated to $\PGL_2(K_\infty)$. Let $V(\sT):=\GL_2(K_\infty)/K_\infty^\times \GL_2(O_\infty)$ be its set of vertices and $E(\sT):=\GL_2(K_\infty)/K_\infty^\times \Ical_\infty$ its set of oriented edges, where $\Ical_\infty$ is the Iwahori subgroup
$$\Ical_\infty:=\left\{ \begin{pmatrix} a&b \\ c & d \end{pmatrix} \in \GL_2(O_\infty)\ \Bigg| \ c  \equiv 0 \bmod \pi_\infty \right\}.$$
For an edge $e$, we denote by $o(e)$ is origin, $t(e)$ its terminus, and $\bar{e}$ the opposite edge. Given $g \in \GL_2(K_\infty)$ let $e_g$ be the coset of $g$ in $E(\sT)$, i.e. the  oriented edge corresponding to $g$ on~$\sT$. More precisely we have 
$$o(e_g) := g \cdot K_\infty^\times \GL_2(O_\infty) \ \in V(\sT)$$
and
\begin{align*}
t(e_g) := g\begin{pmatrix}0 & 1 \\ \pi_\infty & 0 \end{pmatrix} \cdot K_\infty^\times \GL_2(O_\infty)  = g \begin{pmatrix}\theta & 0 \\ 0 & 1 \end{pmatrix} \cdot K_\infty^\times \GL_2(O_\infty)\ \in V(\sT).
\end{align*}
In particular, the opposite edge $\bar{e}_g$ of $e_g$ is represented by $g \begin{pmatrix}0 & 1 \\ \pi_\infty & 0 \end{pmatrix} \in \GL_2(K_\infty)$.

\section{Congruence subgroups \texorpdfstring{$\Gamma_0(\nfk)$}{Gamma0N} and quotient graphs}\label{sec: Quo-G}

Let $\Gamma:=\GL_2(A) \subset \GL_2(K_\infty)$, which acts from the left on $\sT$.
In this section, we shall recall the needed properties of the quotient graphs associated to the congruence subgroups $\Gamma_0(\nfk)$ of $\Gamma$.
Let 
$$\Gamma_\infty:=\left\{\begin{pmatrix} a&b\\ c&d\end{pmatrix} \in \Gamma\ \bigg|\ c = 0\right\}.$$
Recall Weil's decomposition of elements in $\GL_2(K_\infty)$ as follows:

\begin{thm}\label{thm: Weil-dec}
\rm (Cf.\ \cite[3 and 4]{Weil})
\begin{enumerate}
\item Given $g \in \GL_2(K_\infty)$, there exists a unique $r \in \ZZ_{\geq 0}$ such that
$$
g = \gamma \cdot 
\begin{pmatrix}
\theta^r & 0 \\
0 & 1
\end{pmatrix}
\cdot z \cdot \kappa
$$
for some $\gamma \in \Gamma$, $z \in K_\infty^\times$, and $\kappa \in \GL_2(O_\infty)$.
\item
For each $r \in \ZZ_{\geq 0}$, 
let $v_r$ (resp.\ $e_r$) be the vertex (resp.\ oriented edge) of $\sT$ represented by
$\begin{pmatrix}
\theta^r & 0 \\
0 & 1
\end{pmatrix}$.
Then the stabilizer of $v_r$ under $\Gamma$ is
$$\Stab_\Gamma(v_r) = 
\begin{cases}
\GL_2(\FF_q) & \text{ if $r = 0$,} \\
\Gamma_\infty^{(r)} := \left\{\begin{pmatrix}a & b \\ 0 & d \end{pmatrix}\in \Gamma_\infty
\ \bigg| \deg b \leq r \right\} & \text{ if $r > 0$,}
\end{cases}
$$
and the stabilizers of $e_r$ and $\bar{e}_r$ under $\Gamma$ are
$$\forall r \geq 0, \quad\Stab_\Gamma(e_r) = \Stab_\Gamma(\bar{e}_r) = \Gamma_\infty^{(r)}.$$
\end{enumerate}
\end{thm}

Given two vertices $v,v' \in V(\sT)$, we denote by $d(v,v')$ the \emph{distance} between $v$ and $v'$, i.e.\ the number of edges lying in the unique path connecting $v$ and $v'$. 

\begin{lem}\label{lem: Dist}
Given $\gamma = \begin{pmatrix} a&b \\ c&d\end{pmatrix} \in \Gamma$, one has $d(\gamma v_0,v_0) = 2\max(\deg a,\deg b,\deg c,\deg d)$.
\end{lem}

\begin{proof}
By the Iwasawa decomposition, there exists $\kappa \in \GL_2(O_\infty)$ such that
$$
\gamma \cdot \kappa = 
\begin{cases}
\frac{1}{d}\begin{pmatrix} \det \gamma & bd \\ 0 & d^2 \end{pmatrix} & \text{ if $\deg d \geq \deg c$,} \\
\frac{1}{c}\begin{pmatrix} \det \gamma & ac \\ 0 & c^2 \end{pmatrix} &  \text{ otherwise.}
\end{cases}
$$
Note that for a vertex $v \in V(\sT)$ represented by $\begin{pmatrix} \alpha & u \\ 0&\beta\end{pmatrix} \in \GL_2(K_\infty)$, one can derive from \cite[p.\ 185]{Gek} that
$
d(v,v_0) = 
\max\big(\nu_\infty(\alpha)-\nu_\infty(u),0\big) + \big|\min\big(\nu_\infty(\alpha),\nu_\infty(u)\big)-\nu_\infty(\beta)\big|
$. The result then follows from a straightforward argument.
\end{proof}

Given $\nfk \in A_+$, put
$$
\Gamma_0(\nfk) := \left\{ \begin{pmatrix}a&b \\ c&d \end{pmatrix} \in \Gamma \ \bigg|\ c \equiv 0 \bmod \nfk \right\}.
$$
Let $\Gcal(\nfk) := \Gamma_0(\nfk) \backslash \sT$ be the quotient graph of $\sT$ by $\Gamma_0(\nfk)$. Its set of vertices is $V(\Gcal(\nfk)): = \Gamma_0(\nfk)\backslash V(\sT)$ and its set of oriented edges is $E(\Gcal(\nfk)) := \Gamma_0(\nfk)\backslash E(\sT)$. If $e$ is an edge of $\sT$, we denote by $[e]$ the corresponding edge of $\Gcal(\nfk)$.

By Theorem \ref{thm: Weil-dec},
the quotient graph $\Gcal(1)$ is a half line (cf.\ Figure \ref{figure: N=1}), 
and the vertices (resp.\ oriented edges) of $\Gcal(1)$ are represented by $v_r \in V(\sT)$ (resp.\ $e_r$ and $\bar{e}_r$ in $E(\sT)$) for $r \in \ZZ_{\geq 0}$.

\begin{figure}
\begin{tikzpicture}[->, >=stealth, semithick, node distance=1.5cm, inner sep=.5mm, vertex/.style={circle, fill=black}]

\node[vertex] (0) [label=below:${[v_0]}$]{};
  \node[vertex] (1) [right of=0, label=below:${[v_1]}$] {}; 
  \node[vertex] (2) [right of=1, label=below:${[v_2]}$] {};
  \node[vertex] (3) [right of=2, label=below:${[v_3]}$] {};
  \node[] (4) [right of=3] {};

\path[]
    (0) edge  (1) (1) edge (2) (2) edge (3) (3) edge[dashed] (4);   
\end{tikzpicture}
\caption{Graph of $\Gamma \backslash \sT$}\label{figure: N=1}
\end{figure}

For general $\nfk$, by Theorem \ref{thm: Weil-dec} we know that the vertices and the oriented edges of $\Gcal(\nfk)$ can be respectively represented in $\sT$ by elements in
\begin{equation}\label{eq: Rep}
\{\gamma \, v_r\ |\ r \in \ZZ_{\geq 0} \text{ and } \gamma \in \Gamma_0(\nfk) \backslash \Gamma \}  \quad \text{and} \quad
\{\gamma \, e_r,\ \gamma \, \bar{e}_r\ |\ r \in \ZZ_{\geq 0} \text{ and } \gamma \in \Gamma_0(\nfk) \backslash \Gamma \}.
\end{equation}
Moreover, $\gamma \, v_r$ and $\gamma'\, v_{r'}$ (resp.\ $\gamma e_r$ and $\gamma' e_{r'}$) represent the same vertex (resp.\ edge) in $\Gcal(\nfk)$ if and only if $r = r'$ and 
\begin{eqnarray}\label{eq: stab}
&& \gamma'\Stab_\Gamma(v_r)\gamma^{-1} \cap \Gamma_0(\nfk) \neq \emptyset
\quad (\text{resp.\ }
\gamma'\Stab_\Gamma(e_r)\gamma^{-1} \cap \Gamma_0(\nfk) \neq \emptyset). 
\end{eqnarray}

\begin{rem}\label{rem: Quo-G}
${}$
\begin{enumerate}
\item
For $\gamma,\gamma' \in \Gamma$ and distinct $r,r' \in \ZZ_{\geq 0}$, the edge $[\gamma e_r]$ is always different from $[\gamma'  e_{r'}]$ and $[\gamma' \bar{e}_{r'}]$ in $E(\Gcal(\nfk))$. 
\item If $[\gamma \,e_r] = [\gamma'\, e_r]$ in $E(\Gcal(\nfk))$ for some $r \geq 1$ and $\gamma,\gamma'\in\Gamma$, 
then $[\gamma \,e_{r+n}] =[ \gamma' \,e_{r+n}]$ for all $n \in \ZZ_{\geq 0}$.
\end{enumerate}
\end{rem}

For $\gamma = \begin{pmatrix}a&b\\ c&d\end{pmatrix} \in \Gamma$, we let
$$\nfk_\gamma := \frac{\nfk}{\gcd(c^2,\nfk)},$$
and call it the \emph{width} of $\gamma$.

Let $\PP^1(A/\nfk)$ be the projective line over the ring $A/\nfk$ consisting of elements denoted by $(c:d) \bmod \nfk$. The group $\Gamma$ acts from the right on $\PP^{1}(A/\nfk)$. We call $\Ccal(\nfk):= \PP^1(A/\nfk)/\Gamma_\infty$ the set of \emph{cusps of $\Gcal(\nfk)$}.
The $\Gamma_\infty$-orbit represented by $(c:d) \bmod \nfk$ is denoted by $[c:d]$.
The quotient graph $\Gcal(\nfk)$ can be decomposed as follows:

\begin{thm}\label{thm: Quo-G-dec}
Given $\nfk \in A_+$, the quotient graph $\Gcal(\nfk)$ is the union of a finite graph $\Gcal(\nfk)^o$ and a set of ends $E_s$ indexed by the cusps $s \in \Ccal(\nfk)$.
Here: 
\begin{itemize}
\item The set of vertices of the finite subgraph $\Gcal(\nfk)^o$ is the image of 
$$
\{\gamma \, v_r\ |\ 0 \leq r \leq \deg \nfk_{\gamma} - 1 \text{ and } \gamma \in \Gamma_0(\nfk) \backslash \Gamma \} \subset V(\sT),
$$
and the set of edges of $\Gcal(\nfk)^o$ is the image of 
$$
\{\gamma \, e_r, \ \gamma \, \bar{e}_r \ |\ 0 \leq r \leq \deg \nfk_{\gamma} - 2 \text{ and } \gamma \in \Gamma_0(\nfk) \backslash \Gamma \} \subset E(\sT).
$$
\item For each $s = [c:d] \in \Ccal(\nfk)$, we may assume that $\gcd(c,d,\nfk)=1$ and choose $a, b \in A$ so that $ad - bc = 1$; let 
$\gamma_s := \begin{pmatrix} a&b \\ c&d \end{pmatrix} \in \Gamma$ and $\ell_s := \max(0,\deg \nfk_{\gamma_s}-1)$.
The vertices (resp.\ oriented edges) of the end $E_s$ are represented by
$$
\{\gamma_s \, v_r\ |\ r \geq \ell_s\}
\quad (\text{resp.\ } \{\gamma_s \, e_r, \gamma_s \, \bar{e}_r \ |\ r \geq \ell_s\}).
$$
\end{itemize}
\end{thm}

\begin{proof}
See \cite[Section~1.8]{G-N} and \cite[Section~3.1]{T-R}, except for the input of the width $\nfk_\gamma$ for $\gamma \in \Gamma$. We recall the argument here for the sake of completeness.
We first identify $\Gamma_0(\nfk)\backslash \Gamma$ with $\PP^1(A/\nfk)$ by sending $\gamma = \begin{pmatrix}a&b \\ c&d\end{pmatrix}$ to $(c:d) \bmod \nfk$.
In particular, one may take the representatives $\gamma = \begin{pmatrix} a&b \\ c&d \end{pmatrix}$ for the right cosets of $\Gamma_0(\nfk)$ in $\Gamma$ 
satisfying $c \mid \nfk$ and $\deg d \leq \deg \nfk -1$.
For two cosets of $\Gamma_0(\nfk)$ represented by $\gamma = \begin{pmatrix}a&b \\ c&d\end{pmatrix}$ and $\gamma' = \begin{pmatrix}a'&b' \\ c'&d'\end{pmatrix}$ respectively, with $c,c' \mid \nfk$ and $\deg d, \deg d' \leq \deg \nfk -1$,
let $r = \max(0,\deg \nfk_{\gamma} -1)$ and $r' = \max(0,\deg \nfk_{\gamma'}-1)$.
Using \eqref{eq: stab}, Theorem~\ref{thm: Weil-dec} and Remark~\ref{rem: Quo-G}, it can be checked that the edges $\gamma \, e_r$ and $\gamma' \, e_{r'}$ of $\sT$ represent the same edge in $\Gcal(\nfk)$ if and only if there exists $\beta \in \Gamma_\infty$ such that 
$$(c':d') \equiv (c:d)\cdot \beta \ \ \bmod \nfk \quad \in \PP^1(A/\nfk).$$ 
In this case, we have $\nfk_\gamma = \nfk_{\gamma'}$, $r = r'$, and the edges $\gamma e_{r+n}$ and $\gamma' e_{r+n}$ represent the same edge in $\Gcal(\nfk)$ for every $n \in \ZZ_{\geq 0}$. Therefore the result follows.
\end{proof}

\begin{rem}
Algorithmic procedures to compute the quotient graph $\Gamma_0(\nfk) \backslash \sT$ given $\nfk \in A_+$ have been provided in \cite{N}, \cite{G-N}, \cite{T-R} and \cite{B}.
\end{rem}

Let $w_\nfk = \begin{pmatrix}0&-1\\\nfk&0\end{pmatrix} \in \GL_2(K)$. We end up this section by the following technical lemma:

\begin{lem}\label{lem: Key-Lem}
Suppose an element $\gamma = \begin{pmatrix}a&b \\ c&d\end{pmatrix} \in \Gamma$ is given.
\begin{enumerate}
\item\label{lem: Key-Lem1} Let $\ell = \max(\deg c, \deg d)$, and put $\epsilon = 1$ if $\deg c \geq \deg d$ and $0$ otherwise.
There exists $u \in K_\infty$ such that
the edge $[\gamma e_0] \in E(\Gcal(\nfk))$ can be represented by
$$ 
\begin{pmatrix}\pi_\infty^{2 \ell + \epsilon} & u \\ 0&1 \end{pmatrix} \begin{pmatrix}0&1 \\ \pi_\infty &0\end{pmatrix}^\epsilon.
$$
\item\label{lem: Key-Lem2}
Take $x,y \in A$ with $\gcd(x,y) = 1=\gcd(\nfk,cx+dy) $. Let $\delta = \max(\deg x, \deg y)$, and put $\epsilon = 0$ if $\deg x > \deg y$ and $1$ otherwise. There exists $u \in K_\infty$ such that the edge $[\gamma e_0] \in E(\Gcal(\nfk))$ can be represented by
$$ 
w_\nfk
\begin{pmatrix}\pi_\infty^{\deg \nfk + 2 \delta+\epsilon} & u \\ 0&1 \end{pmatrix} \begin{pmatrix}0&1\\ \pi_\infty&0\end{pmatrix}^\epsilon.$$
\end{enumerate}
\end{lem}

\begin{proof}
We may assume $\det \gamma =1$ without loss of generality.
Then \eqref{lem: Key-Lem1} directly follows from the Iwasawa decomposition:
$$
\begin{pmatrix}a&b \\ c&d \end{pmatrix}
=
\begin{cases}
\begin{pmatrix} d^{-2} & b/d \\ 0 & 1 \end{pmatrix}
\begin{pmatrix} d & 0 \\ 0 & d \end{pmatrix} \begin{pmatrix} 1 & 0 \\ d^{-1} c & 1\end{pmatrix} & \text{ if $\deg c < \deg d$,} \\
\begin{pmatrix} c^{-2}\theta^{-1} & a/c \\ 0 & 1 \end{pmatrix} \begin{pmatrix} 0 & \theta \\ 1 & 0 \end{pmatrix} \begin{pmatrix} c&0\\ 0&c\end{pmatrix} \begin{pmatrix} 1 & c^{-1} d \\ 0&-1\end{pmatrix} & \text{ if $\deg c \geq \deg d$,}
\end{cases}
$$
where $\begin{pmatrix} 1 & 0 \\ d^{-1} c & 1\end{pmatrix}$ (resp. $\begin{pmatrix} 1 & c^{-1} d \\ 0&-1\end{pmatrix}$) belongs to $ \Ical_\infty$ in the first (resp. second) case.
For \eqref{lem: Key-Lem2}, 
it is observed that $\text{gcd}(ax+by,cx+dy)$ divides $\text{gcd}(x,y)$, which is equal to $1$. Therefore we have $\gcd(\nfk(ax+by),cx+dy)=1$. Take $\alpha,\beta \in A$ so that
$\alpha \nfk (ax+by) + \beta (cx+dy) = 1$.
Then $$\gamma_0 := \begin{pmatrix} -(cx+dy)& ax+by \\ \alpha \nfk & \beta \end{pmatrix} \in \Gamma_0(\nfk),$$
and $\gamma_0 \gamma$ is equal to
$$
\begin{pmatrix}0&-1\\ \nfk&0\end{pmatrix}
\begin{pmatrix}\nfk^{-1} & \nfk^{-1}(\alpha b \nfk + \beta d) \\ 0& -x\end{pmatrix} \begin{pmatrix} x^{-1} & 0 \\ 0 & 1\end{pmatrix} \begin{pmatrix} 1 & 0 \\ -x^{-1} y & 1 \end{pmatrix}
$$
where $\begin{pmatrix} 1 & 0 \\ -x^{-1} y & 1 \end{pmatrix} \in \Ical_\infty$ if $\deg x > \deg y$, and
$$
\begin{pmatrix}0&-1\\ \nfk&0\end{pmatrix}
\begin{pmatrix}\nfk^{-1} & \nfk^{-1}(\alpha a \nfk + \beta c) \\ 0& y\end{pmatrix} \begin{pmatrix} y^{-1}\theta^{-1} & 0 \\ 0 & 1\end{pmatrix}
\begin{pmatrix} 0&\theta \\ 1&0\end{pmatrix} \begin{pmatrix} 1 & -y^{-1} x \\ 0 & 1 \end{pmatrix}
$$
where $\begin{pmatrix} 1 & -y^{-1} x \\ 0 & 1 \end{pmatrix} \in \Ical_\infty$ if $\deg x \leq \deg y$.
Take
$$
u := \begin{cases}
-(x\nfk)^{-1}(\alpha b\nfk+\beta d) & \text{ if $\deg x > \deg y$,} \\
(y\nfk)^{-1}(\alpha a\nfk + \beta c) & \text{ otherwise.}
\end{cases}
$$
Then the edge $[\gamma_0 \, \gamma\, e_0]$ of $\Gcal(\nfk)$ can be represented by
$$ 
w_\nfk
\begin{pmatrix}\pi_\infty^{\deg \nfk + 2 \delta+\epsilon} & u \\ 0&1 \end{pmatrix}  \begin{pmatrix}0&1\\ \pi_\infty&0\end{pmatrix}^\epsilon.$$
Thus the result holds.
\end{proof}

\section{Harmonic cochains and Fourier expansion}\label{Sec-Har}

We recall the definition and the needed properties of  harmonic cochains on $\sT$.

\subsection{Harmonic cochains}

\begin{defn}
A $\CC$-valued function $f$ on $E(\sT)$ is called a \emph{harmonic cochain} if $f$ satisfies the following \emph{harmonicity property}:
$$\forall e \in E(\sT)\  \ \forall v \in V(\sT), \quad f(e)+f(\bar{e}) = 0 = \sum_{\subfrac{e_v \in E(\sT)}{o(e_v) = v}} f(e_v).$$
If $G$ is a subgroup of $\Gamma$, we say that $f$ is \emph{$G$-invariant} if \[
 \forall \gamma \in G \ \ \forall e \in E(\sT),\quad f(\gamma e ) = f(e).
\]

For $\nfk \in A_+$, let $\Hcal(\nfk)$ be the space of $\Gamma_0(\nfk)$-invariant $\CC$-valued harmonic cochains. An element of $\Hcal(\nfk)$ can be seen as a $\CC$-valued function on $E(\Gcal(\nfk))$.
We call $f$ \emph{cuspidal} if $f$ is finitely supported as a $\CC$-valued function on $E(\Gcal(\nfk))$. The subspace of cuspidal harmonic cochains in $\Hcal(\nfk)$ is denoted by $\Hcal_0(\nfk)$.
\end{defn}

\begin{rem}\label{rem: Har}
Given $\nfk \in A_+$, it is known that:
\begin{enumerate}
\item\label{rem: HarSuppGO} Every $f \in \Hcal_0(\nfk)$ is supported on the finite graph $E(\Gcal(\nfk)^o)$ by harmonicity and Theorem~\ref{thm: Quo-G-dec}.
\item $\dim_\CC \Hcal_0(\nfk)$ is equal to $g(\Gcal(\nfk))$, the genus of the graph $\Gcal(\nfk)$ (cf. \cite[3.2.5]{G-R}, and \cite[Th.~2.17]{G-N} for a formula for this genus).
\item For each cusp $s \in \Ccal(\nfk)$, choose $\gamma_s \in \Gamma$ and $\ell_s \in \ZZ_{\geq 0}$ as in Theorem~\ref{thm: Quo-G-dec}. 
Then we have the following exact sequence (cf.\ \cite[p.~277]{Tei}):
$$
\xymatrix{
0 \ar[r] & \Hcal_0(\nfk) \ar[r] & \Hcal(\nfk) \ar[r]^-{c} & \displaystyle\prod_{[0:1] \neq s \in \Ccal(\nfk)} \CC \ar[r]&  0,
}
$$

where $c(f) := \big(f(\gamma_s e_{\ell_s})\mid [0:1] \neq s \in \Ccal(\nfk)\big)$.
In particular, 
$$\dim_\CC \Hcal(\nfk) = g(\Gcal(\nfk)) + \#(\Ccal(\nfk)) -1.$$
\end{enumerate}
\end{rem}

The following result will be a key-lemma for proving our Sturm-type bounds.
\begin{lem}\label{lem: UD-Lem}
Every harmonic cochain in $\Hcal(\nfk)$ (resp.\ $\Hcal_0(\nfk)$) is uniquely determined by its values at the edges $\gamma e_0$ for all $\gamma \in \Gamma_0(\nfk)\backslash \Gamma$ (resp.\ with $\deg \nfk_\gamma \geq 2$).
\end{lem}
\begin{proof}
For $\Hcal(\nfk)$, this can be derived from \cite[2.13]{G-N} or similarly by combining \eqref{eq: Rep}, \eqref{eq: stab} and the harmonicity property. Moreover assume that $\deg \nfk_\gamma <2$. By Theorem~\ref{thm: Quo-G-dec}, the edge $[\gamma e_0]$ does not belong to $E(\Gcal(\nfk)^o)$ hence it belongs to an end of $\Gcal(\nfk)$. Any cuspidal $\Gamma_0(\nfk)$-invariant harmonic cochain vanishes on it by Remark~\ref{rem: Har} \eqref{rem: HarSuppGO}. This proves the result for $\Hcal_0(\nfk)$. See also \cite[Prop.~3.2]{G-N} for a related statement.
\end{proof}

For each divisor $\mfk$ of $\nfk$, we recall the \emph{Atkin-Lehner involution} $W_\mfk$ on $f \in \Hcal(\nfk)$ which is defined by
$$\forall e \in E(\sT),\quad (f|W_{\mfk})(e):= f\left(\begin{pmatrix} s \mfk & t \\ u\nfk & v\mfk \end{pmatrix} e\right),$$
where $s,t,u,v \in A$ with $sv \mfk^2 - ut \nfk = \mfk $. Note that the operator $W_\mfk$ is independent of the chosen $s,t,u,v$. In the particular the involution $W_\nfk$ on $\Hcal(\nfk)$ is defined by
$$\forall e \in E(\sT),\quad (f|W_{\nfk})(e):= f\left(w_\nfk  e\right)= f\left(\begin{pmatrix}0&-1\\\nfk&0\end{pmatrix}e\right).$$

\subsection{Fourier expansion}\label{Sec-Fourier}

Let $\psi: K_\infty \rightarrow \CC^\times$ be the additive character defined by
$$\psi\left(\sum_n a_n \pi_\infty^n\right) := \exp\left(\frac{2 \pi \sqrt{-1}}{p} \text{Trace}_{\FF_q/\FF_p}(a_1)\right),$$
where $p$ denotes the characteristic of $\FF_q$.
In particular, the ring $A$ is self-dual with respect to $\psi$, i.e. $A^\vee := \{x \in K_\infty \mid \forall a \in A, \psi(ax) = 1 \} = A$. 

Let $f$ be a $\Gamma_\infty$-invariant $\CC$-valued harmonic cochain. Viewing $f$ as a $\CC$-valued function on $\GL_2(K_\infty)$,
the Fourier expansion of $f$
is given by (cf.\ \cite[Chapter III]{Weil2}):
$$\forall r \in \ZZ \ \ \forall u \in K_\infty, \quad f\begin{pmatrix} \pi_\infty^r & u \\ 0&1\end{pmatrix} = \sum_{m \in A} f^*(r,m) \psi(mu)$$
where
$$
f^*(r,m) := \int_{A\backslash K_\infty} f\begin{pmatrix} \pi_\infty^r & u \\ 0 & 1\end{pmatrix} \psi(- mu) du,
$$
and the Haar measure $du$ is chosen to be self-dual with respect to $\psi$, i.e.\ $\text{vol}(A\backslash K_\infty,du) = 1$. Let $$c_0(f):= f^*(2,0)$$ and $$\forall \mfk \in A_+,\quad c_\mfk(f):= |\mfk|_\infty \cdot f^*(\deg \mfk+2,\mfk)$$
(this normalization differs from \cite{Gek2}). The harmonicity property implies the following properties on the Fourier coefficients (cf.\ \cite[Section~2]{Gek-Imp}, \cite[Section~2]{R-T}):

\begin{prop}
Let $f$ be a $\Gamma_\infty$-invariant $\CC$-valued harmonic cochain. For $m \in A$ we have 
\begin{enumerate}
    \item $f^*(r,m) = 0$ unless $r \geq \deg m +2$.
    \item $f^*(\deg m+2+\ell, \varepsilon m) = q^{-\ell} f^*(\deg m+2,m)$ for all $\ell \in \ZZ_{\geq 0}$ and $\varepsilon \in \FF_q^\times$.
    \item 
$f$ is identically zero if and only if $c_0(f)=0$ and for every $\mfk \in A_+$, $c_\mfk(f) = 0$.
\end{enumerate}
\end{prop}

In particular, given a $\Gamma_\infty$-invariant $\CC$-valued harmonic cochain $f$, the Fourier expansion of $f$ can be written as:
\begin{eqnarray}\label{eqn: Four-exp}
\forall r\in\ZZ, \forall u\in K_\infty,\quad
f\begin{pmatrix} \pi_\infty^r & u \\ 0&1\end{pmatrix}
= q^{-r+2}\cdot \left(
c_0(f) + \sum_{\subfrac{\mfk \in A_+}{\deg \mfk +2 \leq r}} c_\mfk(f) \Psi(\mfk u)\right),
\end{eqnarray}
where $\Psi(x) := \sum_{\varepsilon \in \FF_q^{\times}} \psi(\varepsilon x) \in \{-1,q-1\}$.

\begin{rem}\label{rem: cOcusp}
Every $f \in \Hcal_0(\nfk)$ satisfies $c_0(f)=0$. Indeed $f$ is supported on $E(\Gcal(\nfk)^o)$ by Remark~\ref{rem: Har} \eqref{rem: HarSuppGO}, $[e_0] \notin E(\Gcal(\nfk)^o)$ by Theorem~\ref{thm: Quo-G-dec}, and $c_0(f) = q^{-2} f(e_0)$ by \eqref{eqn: Four-exp}.
\end{rem}

\section{Sturm-type bound for harmonic cochains}\label{Sec: Sbound}

The aim of this section is to find a Sturm-type bound for harmonic cochains in $\Hcal(\nfk)$ when a level $\nfk \in A_+$ is given.

\subsection{The cuspidal case}\label{subsec: Cusp}

\subsubsection{The general bound}
By Lemma~\ref{lem: UD-Lem}, any given $f \in \Hcal_0(\nfk)$ is uniquely determined by its values at $\gamma e_0$ for all $\gamma \in \Gamma_0(\nfk) \backslash \Gamma$ with $\deg \nfk_\gamma \geq 2$.
Without loss of generality, we may assume $\gamma = \begin{pmatrix} a & b \\ c&d \end{pmatrix}$ with $c \mid \nfk$ and $\deg d < \deg \nfk$.
By Lemma~\ref{lem: Key-Lem} \eqref{lem: Key-Lem1}, there exists $u \in K_\infty$ such that
$$f(\gamma e_0) = \begin{cases}
f\begin{pmatrix} \pi_\infty^{2\deg d} & u \\ 0&1\end{pmatrix} & \text{ if $\deg c < \deg d$,} \\
- f \begin{pmatrix} \pi_\infty^{2\deg c + 1} & u \\ 0&1\end{pmatrix} & \text{ if $\deg c \geq \deg d$.}
\end{cases}
$$
Since $2 \leq \deg \nfk_\gamma \leq \deg \nfk - \deg c$ and $\deg d < \deg \nfk$, one has that $f$ is uniquely determined by
$$\left\{ f\begin{pmatrix} \pi_\infty^r & u \\ 0& 1\end{pmatrix} |\ 2 \leq r \leq 2\deg \nfk -2 \text{ and } u \in K_\infty \right\}.
$$
From Remark~\ref{rem: cOcusp} and the Fourier expansion \eqref{eqn: Four-exp}, we conclude:

\begin{prop}\label{prop: SBound-0}
Let $\nfk\in A_+$. Then $f \in \Hcal_0(\nfk)$ is identically zero if $c_\mfk(f) = 0$ for all $\mfk \in A_+$ with $\deg \mfk \leq 2\deg \nfk-4$.
\end{prop}

However the bound $2\deg \nfk-4$ seems larger than $\log_q \left( \dim_\CC \Hcal_0(\nfk)\right)$ as $\deg \nfk$ increases. 
In the following, we shall derive a smaller bound using the Fourier expansion with respect to the cusp $[1:0] \in \Ccal(\nfk)$.

\begin{defn}\label{defn: l(n)}
For $c,d \in A$ with $\text{gcd}(c,d) = 1$, let 
$$
\delta_\nfk(c,d):= \min \{\max(\deg x, \deg y)\mid \text{gcd}(cx+dy,\nfk) = 1\}.$$
Set $\epsilon_\nfk(c,d) := 0$ if there are $x_0,y_0 \in A$ satisfying 
$$\deg y_0 < \deg x_0 = \delta_\nfk(c,d) \quad \text{ and } \quad   \text{gcd}(cx_0+dy_0,\nfk) =1,$$ 
and $\epsilon_\nfk(c,d) := 1$ otherwise.
We define
$$\ell(\nfk) := \max\big\{2\delta_\nfk(c,d)+\epsilon_\nfk(c,d)\ \big|\ [c:d] \in \Ccal(\nfk)\big\}.$$
\end{defn}

\begin{lem}\label{lem: Key-Lem.2}
Given $\gamma = \begin{pmatrix} a&b \\ c&d\end{pmatrix} \in \Gamma$, there exists $u \in K_\infty$ such that the edge $[\gamma e_0] \in E(\Gcal(\nfk))$ is represented by
$$ w_\nfk \begin{pmatrix} \pi_\infty^{\deg \nfk + 2 \delta_\nfk(c,d) + \epsilon_\nfk(c,d)} & u \\ 0&1\end{pmatrix} \begin{pmatrix} 0 & 1 \\ \pi_\infty & 0 \end{pmatrix}^{\epsilon_\nfk(c,d)}.$$
\end{lem}

\begin{proof}
Take $x_0,y_0 \in A$ with $\max(\deg x_0,\deg y_0) = \delta_\nfk(c,d)$ and $\text{gcd}(cx_0+dy_0,\nfk) = 1$.
We must have $\text{gcd}(x_0,y_0) = 1$. Then the result follows directly from Lemma~\ref{lem: Key-Lem}.
\end{proof}

\begin{rem}\label{rem: min-dis}
Suppose $\deg \nfk > 0$.
For $\gamma = \begin{pmatrix} a&b \\ c&d\end{pmatrix} \in \Gamma$,
by Lemma~\ref{lem: Dist} we can actually show that 
$\min\big\{d(\gamma_0 \gamma v_0, w_\nfk v_0)\mid \gamma_0 \in \Gamma_0(\nfk)\big\} = \deg \nfk + 2 \delta_\nfk(c,d)-1$. In other words, the minimal distance between the vertices $[\gamma v_0]$ and $[w_\nfk v_0]$ in the quotient graph $\Gcal(\nfk)$ is $\deg \nfk + 2 \delta_\nfk(c,d)-1$.
\end{rem}

We then have:

\begin{prop}\label{prop: Sbound.1}
Let $\nfk \in A_+$. Then $f \in \Hcal_0(\nfk)$ is identically zero if $c_\mfk(f) = 0$ for all $\mfk \in A_+$ with $\deg \mfk \leq \deg \nfk -2 +\ell(\nfk)$. 
\end{prop}

\begin{proof}
Given $f \in \Hcal_0(\nfk)$ satisfying $c_\mfk(f) = 0$ for all $\mfk \in A_+$ with $\deg \mfk \leq \deg \nfk -2 + \ell(\nfk)$,
let $f' := f|W_\nfk$, which belongs to $\Hcal_0(\nfk)$.
From the Fourier expansion \eqref{eqn: Four-exp}, we know that 
$$\forall u \in K_\infty \text{ and } r \leq  \deg \nfk + \ell(\nfk),\quad (f'|W_\nfk)\begin{pmatrix} \pi_\infty^r & u \\ 0 & 1\end{pmatrix} = f \begin{pmatrix} \pi_\infty^r & u \\ 0 & 1\end{pmatrix} =  0.$$
Since $f'$ is uniquely determined by its values at $\gamma e_0$ for $\gamma \in \Gamma_0(\nfk) \backslash \Gamma$ by Lemma~\ref{lem: UD-Lem}, Lemma~\ref{lem: Key-Lem.2} implies that $f' = f|W_\nfk$ is identically zero, and so is $f$.
\end{proof}

Next we shall connect the integer $\ell(\nfk)$ with the number of prime factors of $\nfk$.
Given $m,n \in \ZZ_{\geq 0}$, for each pair $(c,d)$ with $c,d \in A$ and $\text{gcd}(c,d) = 1$, we put
$$\Scal(c,d;m) := \{xc+yd\mid \deg x,\deg y \leq m\}.$$
Let
$$t(c,d;m):=\max\{ \#(\Scal') \mid \Scal' \subset \Scal(c,d;m) \text{ with, for any distinct } \alpha, \beta \in \Scal', \text{gcd}(\alpha,\beta) = 1\}.$$
We define
$$t(m,n):= \min\{t(c,d;m)\mid c,d \in A \text{ with } \text{gcd}(c,d) = 1 \text{ and } m+1 < \max(\deg c,\deg d) < n\}$$
if $n \geq m+3$, and $t(m,n) := + \infty$ otherwise.
Finally, given $\nfk \in A_+$ let $t(\nfk)$ be the number of prime factors of $\nfk$.

\begin{defn}\label{defn: d(n)}
For $\nfk \in A_+$,
put
$$\tau(\nfk):= 
\min\{m \in \ZZ_{\geq 0} \mid t(\nfk)<t(m, \deg \nfk)\}.
$$
\end{defn}

Then:

\begin{cor}\label{cor: SBound.2}
For each $\nfk \in A_+$, we have
$$\ell(\nfk) \leq 2\tau(\nfk)+1.$$
Thus $f \in \Hcal_0(\nfk)$ is identically zero if $c_\mfk(f) = 0$ for $\mfk \in A_+$ with $\deg \mfk \leq \deg \nfk -1 +2\tau(\nfk)$.
\end{cor}

\begin{proof}
Given $\gamma = \begin{pmatrix} a&b \\ c&d\end{pmatrix} \in \Gamma$, we may assume that $c \mid \nfk$ and $\deg c, \deg d < \deg \nfk$.
By Proposition~\ref{prop: Sbound.1}, it suffices to show that $\delta_\nfk(c,d) \leq \tau(\nfk)$. Let $m:=\tau(\nfk)$.
If $\max(\deg c,\deg d) \leq m+1$, then there exists $x,y \in A$ with $\deg x, \deg y \leq m$ such that $cx+dy = 1$. Thus $\delta_\nfk(c,d)\leq m$.
Suppose $m+1<\max(\deg c,\deg d)< \deg \nfk$.
Take $\Scal' \subset \Scal(c,d;m)$ with $\#(\Scal') \geq t(m,\deg \nfk)$ and  $\text{gcd}(\alpha,\beta) =1$ for distinct $\alpha,\beta \in \Scal'$.
Then 
$\text{gcd}(\alpha,\nfk)$ and $\text{gcd}(\beta,\nfk)$ must be relatively prime for distinct $\alpha,\beta \in \Scal'$.
Since $t(\nfk)< t(m,\deg \nfk) \leq \#(\Scal')$, the pigeonhole principle ensures that there exists $\alpha_0 \in \Scal'$ such that $\text{gcd}(\alpha_0,\nfk) = 1$.
Writing $\alpha_0$ as $x_0c+y_0d$ where $x_0,y_0 \in A$ with $\deg x_0,\deg y_0 \leq m$,
we then have
\begin{eqnarray}
\delta_\nfk(c,d) &=& 
\min \{\max(\deg x, \deg y)\mid \text{gcd}(cx+dy,\nfk) = 1\} \nonumber \\
&\leq & \max(\deg x_0,\deg y_0) \nonumber \\
&\leq &  m . \nonumber
\end{eqnarray}
\end{proof}

One may observe that $t(n-2,n)=+\infty$ for any $n\geq 2$ hence $\tau(\nfk) \leq \deg \nfk -2$ when $\deg \nfk \geq 2$. Moreover, it can be checked that $t(0,n) = q+1$ for $n \geq 3$. Thus $\tau(\nfk)= 0$ if $t(\nfk)<q+1$. We also have:

\begin{lem}\label{lem: tBound}
We have $t(1,n) \geq 2q+1$ for $n \geq 4$ hence
\[
 \tau(\nfk)=1 \quad\text{when}\quad q<t(\nfk)\leq 2q.
\]
Consequently when $q < t(\nfk) \leq  2q$, $f \in \Hcal_0(\nfk)$ is identically zero if $c_\mfk(f) = 0$ for all $\mfk \in A_+$ with $\deg \mfk \leq \deg \nfk + 1$.
\end{lem}

\begin{proof}
It suffices to find $\Scal' \subset \Scal(c,d;1)$ with $\#(\Scal') \geq 2q+1$ for every pair $(c,d) \in A^2$ with $\text{gcd}(c,d) = 1$ and $\max(\deg c,\deg d) \geq 3$.

Given $c,d \in A$ with $\deg c \geq 3$ and $\text{gcd}(c,d) = 1$, for $\varepsilon \in \PP^1(\FF_q)$ put
$$c_\varepsilon =
\begin{cases} c+\varepsilon d & \text{ if $\varepsilon \in \FF_q$,} \\
d & \text{ if $\varepsilon = \infty$.}
\end{cases}
$$
There exists $\varepsilon' \in \PP^1(\FF_q)$ such that $\theta - \varepsilon \nmid c_{\varepsilon'}$ for every $\varepsilon \in \FF_q$.
Without loss of generality, assume $\varepsilon' = \infty$ and $\theta \mid c$ (i.e.\ $c = c_0$ and $d = c_\infty$).

Suppose $\theta-1 \mid c$ (resp.\ $\theta-1 \nmid c$).
Then for $(\varepsilon_0,\varepsilon_1) \in \FF_q^\times \times \FF_q$ (resp.\ $ \FF_q\times \FF_q$),
let
$$z(\varepsilon_0,\varepsilon_1):= 
(\varepsilon_0\theta+\varepsilon_1)c + \begin{cases}
d & \text{ if $\theta-1 \mid c$,}\\
(\theta-1)d & \text{ if $\theta-1 \nmid c$.}
\end{cases}
$$
one has
$$\gcd\big(z(\varepsilon_0,\varepsilon_1),c\big) = 1 = \gcd\big(z(\varepsilon_0,\varepsilon_1),d\big),$$
and for $\beta \in \FF_q^\times$:
\begin{eqnarray}
&& \gcd\big(z(\varepsilon_0,\varepsilon_1),c_\beta \big) \nonumber = 
\begin{cases}
\gcd\Big(\varepsilon_0\theta + \varepsilon_1-\beta^{-1},c_\beta\Big)
& \text{ if $\theta-1 \mid c$,}\\
\gcd\Big((\varepsilon_0\theta + \varepsilon_1-\beta^{-1}(\theta-1)),c_\beta\Big)
& \text{ if $\theta-1 \nmid c$.}
\end{cases}
\nonumber
\end{eqnarray}
On the other hand, for each $\varepsilon \in \FF_q^\times$ with $\theta -\varepsilon \nmid c$, there exists a unique $\beta_\varepsilon \in \FF_q^\times$ such that $\theta - \varepsilon \mid c_{\beta_\varepsilon}$.
Thus for $\varepsilon \in \FF_q^\times$ with $\theta-\varepsilon \nmid c$, one has
$$\text{gcd}\Big((\varepsilon_0\theta + \varepsilon_1)c+d,c_\beta\Big) = \theta - \varepsilon \quad \text{ if and only if }
\beta = \beta_\varepsilon \ \text{ and }
$$
$$
\begin{cases}
\varepsilon_0\varepsilon + \varepsilon_1 - \beta_\varepsilon^{-1} = 0 & \text{ if $\theta-1 \mid c$,}\\
\varepsilon_0\varepsilon + \varepsilon_1 - \beta_\varepsilon^{-1}(\varepsilon-1) = 0 & \text{ if $\theta-1\nmid c$.}
\end{cases}
$$
In this case, $\varepsilon_1$ is uniquely determined by the choices of $\varepsilon_0$ and $\varepsilon$.
Suppose $\theta -1 \mid c$ (resp.\ $\theta-1 \nmid c$). There are at most $q-2$ (resp.\ $q-1$) choices of $\varepsilon \in \FF_q^\times$ so that $\theta -\varepsilon \nmid c$. 
Thus we obtain that there are at least $(q-1)\cdot 2$ (resp.\ $q$) choices of the pair $(\varepsilon_0,\varepsilon_1) \in \FF_q^\times \times \FF_q$ (resp.\ $\FF_q\times \FF_q$) so that
$$\forall \beta \in \FF_q^\times,\quad \text{gcd}\big(z(\varepsilon(\varepsilon_0,\varepsilon_1),c_\beta\big) = 1.$$

Let $$\Scal_1':=\Big\{z(\varepsilon_0,\varepsilon_1)\ \Big|\ \forall \beta \in \PP^1(\FF_q),\ \text{gcd}\big(z(\varepsilon_0,\varepsilon_1),c_\beta\big) = 1 \Big\}.$$
Then given distinct $z(\varepsilon_0,\varepsilon_1),z(\varepsilon_0',\varepsilon_1') \in \Scal_1'$, one has
\begin{align*}
\text{gcd}\Big(z(\varepsilon_0,\varepsilon_1), z(\varepsilon_0', \varepsilon_1')\Big) \nonumber &= \text{gcd} \Big(z(\varepsilon_0,\varepsilon_1), (\varepsilon_0'-\varepsilon_0)\theta+(\varepsilon_1'-\varepsilon_1)\Big)
\\
&= 1.
\end{align*}
Note that
$$\#\Scal_1' \geq \begin{cases}
2(q-1) & \text{ if $\theta-1\mid c$,}\\
q & \text{ if $\theta-1 \nmid c$.}
\end{cases}
$$
Take $$\Scal_0':=\big\{c_\varepsilon\ \big|\ \varepsilon \in \PP^1(\FF_q)\big\} \quad \text{ and } \quad \Scal' := \Scal_0' \cup \Scal_1'.$$
Then $\Scal' \subset \Scal(c,d;1)$ and 
$\#(\Scal') \geq 2q+1$.
Therefore $t(1,n) \geq 2q+1$ for $n \geq 4$.
\end{proof}

\subsubsection{The case of prime power level}\label{subsubsec: prime-power}

Suppose $\nfk = \pfk^r$ where $r$ is a positive integer and $\pfk \in A_+$ is a prime. In this case we are able to give a better bound than Proposition~\ref{prop: Sbound.1}.
Let $\gamma = \begin{pmatrix} a&b \\ c&d\end{pmatrix} \in \Gamma$. Suppose that $\pfk \nmid c$. Take $\alpha,\beta \in A$ such that $\alpha a \nfk + \beta c = 1$. Then $\gamma_0 = \begin{pmatrix} -c &a \\ \alpha \nfk & \beta \end{pmatrix} \in \Gamma_0(\nfk)$ and
$$
\gamma_0 \,\gamma = \begin{pmatrix} 0 & -\det \gamma \\ 1 & \alpha b \nfk + \beta d \end{pmatrix} = \begin{pmatrix} 0&-1 \\ \nfk&0 \end{pmatrix} \begin{pmatrix} \nfk^{-1} & \nfk^{-1}(\alpha b \nfk + \beta d) \\ 0 & - \det \gamma\end{pmatrix} .$$
Thus for $f \in \Hcal_0(\nfk)$, one has
$$f(\gamma e_0) = (f|W_\nfk)\begin{pmatrix} \pi_\infty^{\deg \nfk} & -\det \gamma ^{-1} \nfk^{-1}(\alpha b \nfk + \beta d) \\ 0 & 1\end{pmatrix}.$$
If $\pfk \mid c$, then $\text{gcd}(\varepsilon c+d,\nfk) = 1$ for every $\varepsilon \in \FF_q$.
Similarly, we take $\alpha_{\varepsilon}, \beta_{\varepsilon} \in A$ such that $\alpha_{\varepsilon}(\varepsilon a + b) \nfk + \beta_\varepsilon (\varepsilon c + d) = 1$.
Then $\gamma_{0,\varepsilon} = \begin{pmatrix} -(\varepsilon c + d) & \varepsilon a + b \\ 
\alpha_\varepsilon \nfk & \beta_\varepsilon \end{pmatrix} \in \Gamma_0(\nfk)$ and we have
$$
\gamma_{0,\varepsilon}\, \gamma \begin{pmatrix} \varepsilon & 1 \\ 1&0\end{pmatrix} = \begin{pmatrix} 0 & \det \gamma \\ 1 & \alpha_{\varepsilon} a \nfk + \beta_{\varepsilon} c \end{pmatrix} = \begin{pmatrix} 0&-1 \\ \nfk&0 \end{pmatrix} \begin{pmatrix} \nfk^{-1} & \nfk^{-1}(\alpha_\varepsilon b \nfk + \beta_\varepsilon c) \\ 0 & \det \gamma\end{pmatrix} .$$

For $f \in \Hcal_0(\nfk)$, the harmonicity property implies
$$f(\gamma e_0) = -\sum_{\varepsilon \in \FF_q} 
f\left(\gamma_{0,\varepsilon} \gamma \begin{pmatrix} \varepsilon & 1 \\ 1 & 0 \end{pmatrix}\right)
= -\sum_{\varepsilon \in \FF_q}
(f|W_\nfk)\begin{pmatrix} 
\nfk^{-1} & \nfk^{-1}(\alpha_\varepsilon b \nfk + \beta_\varepsilon c) \\ 0 & \det \gamma\end{pmatrix}.$$
Following a similar argument as in Proposition~\ref{prop: Sbound.1}, we get:

\begin{cor}\label{cor: SBound.3}
Suppose $\nfk = \pfk^r$ where $r$ is a positive integer and $\pfk \in A_+$ is a prime.
Then $f \in \Hcal_0(\nfk)$ is identically zero if $c_\mfk(f) = 0$ for all $\mfk \in A_+$ with $\deg \mfk \leq \deg \nfk - 2$.
\end{cor}

%

\subsubsection{Computational data}
It seems difficult to give a precise formula for $t(m,n)$ when $m>0$ in general. However we are able to compute the actual value using SageMath in the following cases:

\begin{eqnarray}\label{V2}
\begin{tabular}{|c|ccccccc|}
\multicolumn{8}{c}{Value of $t(m,n)$ \ ($q=2$)} \\
\hline 
\diaghead{\theadfont aaaaaaaaaa}%
{$m$}{$n$}  & 4 & 5 & 6 & 7 & 8 & 9 & 10\\
\hline
       1 & 5 & 5 & 5 & 5 & 5 & 5 & 5 \\
       2 & $\infty$ & 11 & 11 & 10 & 9 & 9 & 8 \\
       3 & $\infty$ & $\infty$ & 33 & 30 & 27 & 23 & 23 \\
\hline
\end{tabular}
\end{eqnarray}

\begin{eqnarray} \label{Value of t(m,n) q=3}
\begin{tabular}{|c|ccccc|}
\multicolumn{6}{c}{Value of $t(m,n)$ \ ($q=3$)} \\
\hline 
\diaghead{\theadfont aaaaaaaaaa}%
{$m$}{$n$} & 4 & 5 & 6 & 7 & 8\\
\hline
       1 & 12 & 10 & 10 & 10 & 10\\
       2 & $\infty$ & 64 & 55 & 48 & 43\\
\hline
\end{tabular}
\end{eqnarray}
From these tables, we predict the following lower bound for $t(m,n)$:
$$\text{ for every $n \geq m+3$,} \quad t(m,n) \ ?\!\!\geq (m+1)q+1.$$
If so, then we would get for any $\nfk \in A_+$:
\begin{equation}\label{eq-predict}
\deg \nfk -1 + 2 \tau(\nfk) \ ?\!\! \leq \deg \nfk  -1 + 2 \left\lfloor\frac{t(\nfk)-1}{q}\right\rfloor=: b'(\nfk).
\end{equation}
which is much smaller than the bound $2 \deg \nfk -4$ in Proposition~\ref{prop: SBound-0} when $\deg \nfk$ is large.
It is observed that $t(\nfk) \leq 2q$ when $\deg \nfk \leq 10$ (except for $q = 2$ and $\deg \nfk = 10$), therefore by Lemma~\ref{lem: tBound} and the fact $t(2,10) = 8$, the inequality~\eqref{eq-predict} indeed holds at least for $\nfk \in A_+$ with $\deg \nfk \leq 10$.

\bigskip
The quantity $b'(\nfk)$ is much easier to compute than $\tau(\nfk)$. We now numerically compare it to the optimal Sturm bound for $\Gamma_0(\nfk)$-invariant cuspidal harmonic cochains, which is :
$$b^{\text{true}}(\nfk):=\min\Big\{ b \in \ZZ_{\geq 0} \ \Big|\ \text{$f \in \Hcal_0(\nfk)$ is identically zero if $c_\mfk(f) = 0$ for any $\deg \mfk \leq b$}\Big\}.
$$
First let us explain how we compute $b^{\text{true}}(\nfk)$ through genera of finite subgraphs of $\Gcal(\nfk)^o$. Put $\Hcal_0(\nfk)_{(\ell)}: = \{ f \in \Hcal_0(\nfk) \mid \forall \mfk \in A_+, \deg \mfk \leq \ell, c_{\mfk}(f)=0 \}$. Given $\ell \in \ZZ_+$ and $u\in K_\infty$, let 
\[
e(\ell,u):= \begin{pmatrix} \pi^{\ell+2} & u \\ 0 & 1 \end{pmatrix} e_0 \in E(\sT).
\]
The Fourier expansion \eqref{eqn: Four-exp} shows that, given $\nfk \in A_+$, $f \in \Hcal_0(\nfk)$ and $\ell \in \ZZ_+$, one has :
\[
f \in \Hcal_0(\nfk)_{(\ell)} \text{  if and only if }\quad \forall\, 0 \leq \ell' \leq \ell, \forall\, u \in \pi_\infty O_\infty / \pi_\infty^{\ell'+2} O_\infty, \quad f(e(\ell',u))=0.
\]
Let $\Gcal(\nfk)^o_{(\ell)}$ be the subgraph of $\Gcal(\nfk)^o$ obtained by removing the edges of the form $[e(\ell',u)]$ and $[\bar{e}(\ell',u)]$ for each $0 \leq \ell'\leq \ell$ and $u \in \pi_\infty O_\infty / \pi_\infty^{\ell'+2} O_\infty$. From these observations we get:
\begin{lem}
\begin{enumerate}
    \item Let $f \in \Hcal_0(\nfk)$. Then $f \in \Hcal_0(\nfk)_{(\ell)}$ if and only $f$ is supported on the edges of  $\Gcal(\nfk)^o_{(\ell)}$.
    \item $\dim_\CC \Hcal_0(\nfk)_{(\ell)}$ is equal to the genus   $g(\Gcal(\nfk)^o_{(\ell)})$ of $\Gcal(\nfk)^o_{(\ell)}$.
    \item $b^{\mathrm{true}}(\nfk) = \min \{ \ell \in \ZZ_+ \mid g(\Gcal(\nfk)^o_{(\ell)})=0 \}$.
\end{enumerate}
\end{lem}
With the help of SageMath, we compute values of this genus. For $n \in \ZZ$ with $n \geq 3$ we put
$$b^{\text{true}}(n):= \max\{b^{\text{true}}(\nfk)\mid \deg \nfk = n\}, \quad
b'(n) := \max\{b'(\nfk) \mid \deg \nfk = n\}.$$
Using this method we have obtained the following data which show that our predicted bound $b'(n)$ actually reaches the sharp bound $b^{\text{true}}(n)$ in certain cases.
\begin{eqnarray} \label{compare bounds}
\begin{tabular}{|c|c|c|}
\multicolumn{3}{c}{$q=2$} \\
\hline 
$n$ & $b^{\text{true}}(n)$ & $b'(n)$ \\
\hline
3 & 1 & 2 \\
4 & 3 & 5\\
5 & 5 & 6\\
6 & 6 & 7 \\
7 & 8 & 8 \\
8 & 9 & 9 \\
9 & 10 & 10 \\
10 & 11 & 13\\
\hline
\end{tabular}
\quad \quad 
\begin{tabular}{|c|c|c|}
\multicolumn{3}{c}{$q=3$} \\
\hline 
$n$ & $b^{\text{true}}(n)$& $b'(n)$ \\
\hline
3 & 1 & 2\\
4 & 3 & 3 \\
5 & 4 & 6 \\
6 & 6 & 7 \\
7 & 8 & 8 \\
8 & 9 & 9 \\
9 & 10 & 10 \\
10 & 11 & 11 \\
\hline
\end{tabular}
\end{eqnarray}


\subsection{The non-cuspidal case}\label{subsec: Non-cusp}

Suppose $\nfk \in A_+$ is given.
For $f \in \Hcal(\nfk)$,
we first show that the constant coefficient $c_0(f)$ in the Fourier expansion \eqref{eqn: Four-exp} is uniquely determined by $c_\mfk(f)$ for finitely many $\mfk \in A_+$:

\begin{lem}\label{lem: const}
Suppose $\nfk \in A_+$ is given.
For $f \in \Hcal(\nfk)$ with $c_\mfk(f) = 0$ for all $\deg \mfk < \deg \nfk$, we must have $c_0(f) = 0$.
\end{lem}

\begin{proof}
For $f{|W_\nfk} \in \Hcal(\nfk)$, the harmonicity property gives:
\begin{eqnarray}\label{eqn: non-cusp}
0 &=& (f|W_\nfk)\begin{pmatrix}0&1 \\ 1&1\end{pmatrix} + (f|W_\nfk)\left(\begin{pmatrix} 0&1 \\ 1&1\end{pmatrix}\begin{pmatrix} 0&\theta \\ 1 &0\end{pmatrix}\right) \nonumber \\
&=& f \begin{pmatrix} 1& 1 \\ 0 & \nfk \end{pmatrix}
+ f \left(\begin{pmatrix} 1 & 0 \\ \nfk & 1\end{pmatrix} \begin{pmatrix} 0&-1 \\ \nfk&0\end{pmatrix}
\begin{pmatrix} 1&0 \\ 1 & \theta \end{pmatrix}\right) 
\nonumber \\
&=& f\begin{pmatrix} \pi_\infty^{\deg \nfk} & -\nfk^{-1} \\ 0&1\end{pmatrix} + f\begin{pmatrix} \pi_\infty^{\deg \nfk + 1} & \nfk^{-1} \\ 0&1\end{pmatrix}.
\end{eqnarray}
Suppose $c_\mfk(f) = 0$ for all $\mfk \in A_+$ with $\deg \mfk < \deg \nfk$.
Then the Fourier expansion of $f$ \eqref{eqn: Four-exp} implies 
$$f\begin{pmatrix} \pi_\infty^{\deg \nfk} & -\nfk^{-1} \\ 0&1\end{pmatrix} = q^{-\deg \nfk +2} c_0(f) \quad \text{and} \quad f\begin{pmatrix} \pi_\infty^{\deg \nfk + 1} & \nfk^{-1} \\ 0&1\end{pmatrix}
= q^{-\deg \nfk + 1} c_0(f).
$$
From \eqref{eqn: non-cusp}, we get $c_0(f) = 0$.
\end{proof}

Therefore:

\begin{prop}\label{prop: SBound.non-cusp}
Let $\nfk \in A_+$. Then $f \in \Hcal(\nfk)$ is identically zero if $c_\mfk(f) = 0$ for $\mfk \in A_+$ with $\deg \mfk \leq \deg \nfk - 1 + 2 \tau(\nfk)$.
\end{prop}

\begin{proof}
Let $f \in \Hcal(\nfk)$ with $c_\mfk(f) = 0$ for all $\mfk \in A_+$ such that $\deg \mfk \leq \deg \nfk - 1 + 2 \tau(\nfk)$. By Lemma~\ref{lem: const} we have $c_0(f) = 0$, which shows by the Fourier expansion \eqref{eqn: Four-exp} that
$$\forall r \leq \deg \nfk + 1 + 2 \tau(\nfk),\ \forall u \in K_\infty,\quad f\begin{pmatrix} \pi_\infty^r & u \\ 0&1\end{pmatrix} = 0. 
$$

Let $f':= f|W_\nfk$. Then Lemma~\ref{lem: Key-Lem.2} and Corollary~\ref{cor: SBound.2} implies that
$f'(\gamma e_0) = 0$ for all $\gamma \in \Gamma$.
Therefore $f'$ is identically zero by Lemma~\ref{lem: UD-Lem}, and so is $f$.
\end{proof}

\begin{rem}\label{rem: otherboundsforH}
\hspace{2em}
\begin{enumerate}
 \item As directly seen from the proof, the bound in Proposition~\ref{prop: SBound.non-cusp} can be improved to $\max(\deg \nfk-2+\ell(\nfk),\deg \nfk-1)$.
 \item\label{rem: otherboundsforHcoarse} Similarly to Proposition~\ref{prop: SBound-0}, by combining Lemma~\ref{lem: const} and Lemma \ref{lem: Key-Lem}\eqref{lem: Key-Lem1}, we may obtain for $\Hcal(\nfk)$ the coarse bound $\max (2 \deg \nfk -3,\deg \nfk -1)$ which does not involve $\tau(\nfk)$ nor $\ell(\nfk)$.
\end{enumerate}
\end{rem}

\section{Applications}\label{sec: App}

\subsection{Generators of the Hecke algebra}\label{subsec: Hec-Alg}

Suppose $\nfk \in A_+$ is given.
For $\mfk \in A_+$, the \emph{$\mfk$-th Hecke operator $T_\mfk$ on $\Hcal(\nfk)$} is defined by:
$$\forall e \in E(\sT),\quad(f|T_\mfk)(e) := \sum f\left(\begin{pmatrix} \afk&b \\ 0&\dfk\end{pmatrix}e\right),
$$
where the sum is over $\afk,\dfk \in A_+$, $b \in A$ with $\afk\dfk = \mfk$, $\text{gcd}(\afk,\nfk) = 1$, and $\deg b < \deg \dfk$.
It is known that $T_\mfk$ and $T_{\mfk'}$ commute with each other if $\text{gcd}(\mfk,\mfk') = 1$, and for any prime $\pfk \in A_+$ and $r \in \ZZ_{\geq 0}$, one has
$$T_{\pfk^{r+2}} = T_{\pfk^{r+1}} T_\pfk - \mu_\nfk(\pfk)\cdot |\pfk|_\infty \cdot T_{\pfk^r},$$
where $\mu_\nfk(\pfk) := 1$ if $\pfk \nmid \nfk$ and $0$ otherwise.
Moreover, for $f \in \Hcal(\nfk)$, it can be checked that
$$\forall \mfk \in A_+,\quad c_1(f|T_\mfk) = c_\mfk(f).$$
Let $\TT(\nfk):=\CC[T_\mfk\, |\, \mfk \in A_+] \subset \End_\CC\big(\Hcal(\nfk)\big)$, the Hecke algebra on $\Hcal(\nfk)$.

\begin{lem}\label{lem: perfect-pairing}
We have the following perfect pairing:
$$
\begin{tabular}{crclcc}
$\langle \cdot,\cdot\rangle \ :$ & $\Hcal(\nfk)$ &$\times$ &  $\TT(\nfk)$ & $\longrightarrow$ & $\CC$ \\
& $(\ f$ &, & $T\ )$ & $\longmapsto$ & $c_1(f|T)$
\end{tabular}
$$
which satisfies $\langle f|T,T'\rangle  = \langle f,TT'\rangle $ for all $f \in \Hcal(\nfk)$ and $T,T' \in \TT(\nfk)$.
\end{lem}
\begin{proof}
Adapt Gekeler's proof \cite[Theorem 3.17]{Gek2} in the cuspidal case by using Lemma~\ref{lem: const} and the Fourier expansion~\eqref{eqn: Four-exp}.
\end{proof}
%

From this perfect pairing, Proposition~\ref{prop: SBound.non-cusp} provides a bound for the number of Hecke operators generating $\TT(\nfk)$:

\begin{cor}\label{cor: Hec-gen}
The Hecke algebra $\TT(\nfk)$ is spanned as a $\CC$-vector space by $T_\mfk$ for $\mfk \in A_+$ with $\deg \mfk \leq \deg \nfk - 1 + 2 \tau(\nfk)$.
\end{cor}

\begin{proof}
Let $\TT'$ be the $\CC$-subspace spanned by $T_\mfk$ for $\mfk \in A_+$ with $\deg\mfk \leq \deg \nfk - 1 + 2 \tau(\nfk)$.
Then Proposition~\ref{prop: SBound.non-cusp} shows that the pairing $\langle \cdot,\cdot \rangle$ gives an embedding map from $\Hcal(\nfk)$ to the dual space of $\TT'$.
This implies $\TT' = \TT(\nfk)$ from the perfectness of $\langle \cdot,\cdot\rangle$.
\end{proof}


Note that $\Hcal_0(\nfk)$ is invariant by $\TT(\nfk)$.
Let $\TT_0(\nfk)$ be the image of $\TT(\nfk)$ in $\End_\CC\big(\Hcal_0(\nfk)\big)$ under the restriction map.
The pairing $\langle \cdot,\cdot\rangle$ restricted to $\Hcal_0(\nfk) \times \TT_0(\nfk)$ is still perfect.
Similarly, by Proposition~\ref{prop: Sbound.1} and Corollary~\ref{cor: SBound.3}, we obtain the following result for the cuspidal Hecke algebra:

\begin{cor}
The cuspidal Hecke algebra $\TT_0(\nfk)$ is spanned as a $\CC$-vector space by $T_\mfk$ for $\mfk \in A_+$ with $\deg \mfk \leq \deg \nfk -2 + \ell(\nfk)$, where $\ell(\nfk)$ is defined in \text{\rm Definition~\ref{defn: l(n)}}.
Moreover, if $\nfk$ is a prime power, then $\TT_0(\nfk)$ is spanned by $T_\mfk$ for $\mfk \in A_+$ with $\deg \mfk \leq \deg \nfk - 2$.
\end{cor}

\begin{rem}\label{rem: coarseboundforHecke}
The coarse bounds of Proposition~\ref{prop: SBound-0} and Remark \ref{rem: otherboundsforH}\eqref{rem: otherboundsforHcoarse} for $\Hcal_0(\nfk)$ and $\Hcal(\nfk)$ translate directly into the same bounds for the Hecke algebras $\TT_0(\nfk)$ and $\TT(\nfk)$, respectively.
\end{rem}

\subsection{The case of the \lq\lq new\rq\rq\ subspace}\label{subsec: new-space}

Given $f_1,f_2 \in \Hcal_0(\nfk)$, recall the \emph{Petersson inner product}:
$$\langle f_1, f_2 \rangle_{\text{Pet}} := \sum_{[e] \in E(\Gcal(\nfk))} \frac{f_1(e)\overline{f_2(e)}}{\#\text{Stab}_{\Gamma_0(\nfk)}(e)}$$
where $\overline{\cdot}$ denotes here the complex conjugation. A cuspidal harmonic cochain is called \emph{old} if it is a $\CC$-linear combination of the following type of harmonic cochains:
$$\forall e \in E(\sT),\quad f_{\mfk'}(e) := f\left(\begin{pmatrix}1 & 0 \\ 0 & \mfk' \end{pmatrix} e\right)$$
where $f \in \Hcal_0(\mfk)$ with $\mfk, \mfk' \in A_+$, $(\mfk\cdot \mfk') \mid \nfk$ and $\mfk \neq \nfk$.
Let 
$$\Hcal_0^{\text{new}}(\nfk):=\big\{f \in \Hcal_0(\nfk) \ \big| \text{for all old } f' \in \Hcal_0(\nfk), \ \langle f, f'\rangle_{\text{Pet}} = 0 \big\}.$$

We then have:
\begin{lem}\label{lem: Key-Lem.3}
Given $\nfk \in A_+$, suppose $\nfk$ is square-free.
Then $f \in \Hcal_0(\nfk)$ is identically zero if, for every $u \in K_\infty$ and $\mfk \in A_+$ with $\mfk \mid \nfk$ and $\deg \mfk \leq \deg \nfk-2$,
\begin{eqnarray}\label{eqn: sqr-free}
(f|W_\mfk)\begin{pmatrix} \pi_\infty^{\deg \nfk-\deg \mfk} & u \\ 0 & 1\end{pmatrix} &=& 0.
\end{eqnarray}
\end{lem}

\begin{proof}
Given a coset $\Gamma_0(\nfk) \gamma \in \Gamma_0(\nfk) \backslash \Gamma$, we may assume that the representative
$\gamma$ is of the form $ \begin{pmatrix} a&b \\ \mfk&d\end{pmatrix}$ with $\mfk \mid \nfk$.
Let $\mfk' := \nfk/\mfk$. Since $\nfk$ is square-free, there exist $\alpha, \beta \in A$ with $\alpha a\mfk'+\beta \mfk = 1$.
For $f \in \Hcal_0(\nfk)$, one has
\begin{eqnarray}
(f|W_{\mfk'})(\gamma) &=& f\left(\begin{pmatrix} \alpha \mfk' & \beta \\
-\nfk & a \mfk' \end{pmatrix} \begin{pmatrix} a&b \\ \mfk&d\end{pmatrix}\right) \nonumber \\
&=& f\begin{pmatrix} 1 & \alpha b \mfk' + \beta d \\ 0 & \mfk'\det \gamma  \end{pmatrix}. \nonumber 
\end{eqnarray}

Suppose $f \in \Hcal_0(\nfk)$ satisfies \eqref{eqn: sqr-free}.
Let $f' := f|W_\nfk$. 
For $\gamma = \begin{pmatrix} a&b \\ \mfk&d\end{pmatrix} \in \Gamma$ with $\mfk\mid \nfk$ and $\deg \mfk \leq \deg \nfk-2$, we then have
$$
f'(\gamma) = \Big(\big(f|W_\mfk\big)|W_{\mfk'}\Big)(\gamma) = (f|W_\mfk)\begin{pmatrix}
\pi_\infty^{\deg \mfk'} & (\mfk'\det \gamma)^{-1}(\alpha b \mfk' + \beta d)\\ 0&1 
\end{pmatrix} = 0.
$$
Therefore $f'$ is identically zero by Lemma~\ref{lem: UD-Lem}, and so is $f$.
\end{proof}

Given $\nfk \in A_+$ and a prime factor $\pfk$ of $\nfk$, suppose $\pfk^2 \nmid \nfk$.
Let $f \in \Hcal_0^{\text{new}}(\nfk)$.
It is known that $f|(T_\pfk+W_\pfk) = 0$.
Thus for a square-free factor $\nfk_0$ of $\nfk$ which is coprime to $\nfk/\nfk_0$, one has
$$ f|W_{\nfk_0} = (-1)^{t(\nfk_0)}\cdot f|T_{\nfk_0},$$
where $t(\nfk_0)$ is the number of prime factors of $\nfk_0$.
Moreover, we have:

\begin{lem}\label{lem: 5.5}
Given $\nfk_0, \nfk \in A_+$ with $\nfk_0 \mid \nfk$, put $\nfk_0' := \nfk/\nfk_0$.
Suppose $\nfk_0$ is square-free and coprime to $\nfk_0'$.
For each $u \in K_\infty$, the following identity holds:
$$
(f|W_{\nfk_0})\begin{pmatrix} \pi_\infty^{\deg \nfk_0'} & u \\ 0 & 1\end{pmatrix}
= (-1)^{t(\nfk_0)} q^{-\deg \nfk_0' + 2} 
\sum_{\subfrac{\mfk \in A_+}{\deg \mfk + 2 \leq \deg \nfk_0'}}
c_{\nfk_0 \mfk}(f) \Psi(\mfk u).
$$
\end{lem}

\begin{proof}
The previous discussion tells us that
\begin{eqnarray}
(f|W_{\nfk_0})\begin{pmatrix} \pi_\infty^{\deg \nfk_0'} & u \\ 0&1\end{pmatrix}
&=& (-1)^{t(\nfk_0)} (f|T_{\nfk_0})\begin{pmatrix} \pi_\infty^{\deg \nfk_0'} & u \\ 0&1\end{pmatrix} \nonumber \\
&=& (-1)^{t(\nfk_0)} q^{-\deg \nfk_0'+2}
\sum_{\subfrac{\mfk \in A_+}{\deg \mfk +2\leq \deg \nfk_0'}} \!\!\! c_\mfk(f|T_{\nfk_0}) \, \Psi(\mfk u). \nonumber
\end{eqnarray}
Note that as $\nfk_0 \mid \nfk$, one has
$$\forall \mfk \in A_+, \quad T_{\nfk_0}T_\mfk = T_{\nfk_0 \mfk}.
$$
Thus for $\mfk \in A_+$ with $\deg \mfk+2 \leq \deg \nfk_0'$, we get
$$c_\mfk(f|T_{\nfk_0}) =  c_1(f|T_{\nfk_0}T_\mfk) = c_1(f|T_{\nfk_0 \mfk}) =  c_{\nfk_0\mfk}(f),$$
Therefore the proof is complete.
\end{proof}


When $\nfk$ is square-free, the previous two lemmas give us a smaller bound for $f \in \Hcal_0^{\text{new}}(\nfk)$ than Proposition~\ref{prop: Sbound.1}.

\begin{prop}\label{prop: SBound.newspace}
Given $\nfk \in A_+$, suppose $\nfk$ is square-free. Then $f \in \Hcal_0^{\text{\rm new}}(\nfk)$ is identically zero if $c_\mfk(f) = 0$ for every $\mfk \in A_+$ with $\deg \mfk \leq \deg \nfk -2$.
\end{prop}

\begin{proof}
It suffices to assume that $\deg \nfk \geq 3$ by Remark~\ref{rem: smalldeg}.
Given $\nfk_0 \in A_+$ with $\nfk_0 \mid \nfk$ and $\deg \nfk_0 \leq \deg \nfk -2$, let $\nfk_0' := \nfk/\nfk_0$. For $u \in K_\infty$, by Lemma~\ref{lem: 5.5} one has
$$
(f|W_{\nfk_0})\begin{pmatrix} \pi_\infty^{\deg \nfk_0'} & u \\ 0&1\end{pmatrix}\ =\ 
(-1)^{t(\nfk_0)} q^{-\deg \nfk_0'+2}
\sum_{\subfrac{\mfk \in A_+}{\deg \mfk +2\leq \deg \nfk_0'}} \!\!\! c_{\nfk_0\mfk}(f) \, \Psi(\mfk u) \ =\ 0.
$$
The result then follows from Lemma~\ref{lem: Key-Lem.3}.
\end{proof}


When $\nfk$ is not square-free, we may also go a bit further in the following case:
\begin{lem}\label{lem: p2q}
Let $\nfk = \pfk^2 \qfk$ for two primes $\pfk, \qfk \in A_+$ with $\deg \qfk = 1$.
Given $f \in \Hcal_0^{\text{\rm new}}(\nfk)$, we have that $f$ is identically zero if $c_\mfk(f) = 0$ for every $\mfk \in A_+$ with $\deg \mfk \leq \deg \nfk -2$.
\end{lem}

\begin{proof}
From Corollary~\ref{cor: SBound.3}, we may assume that $\pfk$ and $\qfk$ are distinct.
Let $f \in \Hcal_0^{\text{new}}(\nfk)$ with $c_\mfk(f) = 0$ for every $\mfk \in A_+$ with $\deg \mfk \leq \deg \nfk -2$.
By Lemma~\ref{lem: UD-Lem}, it suffices to show that $(f|W_\nfk)(\gamma) = 0$
for $\gamma \in \Gamma_0(\nfk)\backslash \Gamma$ with $\deg \nfk_\gamma \geq 2$.
We may take $\gamma$ to be of the form
$$
\gamma = \begin{pmatrix} a& b \\ \nfk_0 & d\end{pmatrix}
\quad \text{ with $\nfk_0 \mid \nfk$ and $\deg \nfk_\gamma \geq 2$.}
$$
Then $\nfk_0 = 1$ or $\qfk$.
Put $\nfk_0' := \nfk/\nfk_0$. Then $\gcd(\nfk_0,\nfk_0') = 1$.
Applying the argument in Lemma~\ref{lem: Key-Lem.3}, one has
$$
(f|W_\nfk)(\gamma)
=(f|W_{\nfk_0})\begin{pmatrix}\pi_\infty^{\deg \nfk_0'} & u \\ 0&1\end{pmatrix} \quad
\text{ for some $u \in K_\infty$.}
$$
Since $c_\mfk(f) = 0$ for every $\mfk \in A_+$ with $\deg \mfk \leq \deg \nfk - 2$, Lemma~\ref{lem: 5.5} implies
$$
(f|W_\mfk)\begin{pmatrix}\pi_\infty^{\deg \mfk'} & u \\ 0&1\end{pmatrix} = 0.
$$
Therefore the result holds.
\end{proof}

Note that the subspace $\Hcal_0^{\text{new}}(\nfk)$ is invariant by $\TT_0(\nfk)$.
Let $\TT_0^{\text{new}}(\nfk)$ be the image of $\TT_0(\nfk)$ in $\End_\CC\big(\Hcal_0^{\text{new}}(\nfk)\big)$ under the restriction map.
Then the pairing $\langle \cdot,\cdot\rangle$ restricted to $\Hcal_0^{\text{new}}(\nfk) \times \TT_0^{\text{new}}(\nfk)$ is still perfect.
Consequently, we obtain:

\begin{cor}\label{cor: Hec-Alg.newspace}
Given $\nfk \in A_+$, suppose $\nfk$ is either square-free or $\nfk = \pfk^2 \qfk$ for two primes $\pfk,\qfk \in A_+$ with $\deg \qfk = 1$.
The Hecke algebra $\TT_0^{\text{\rm new}}(\nfk)$ is spanned as a $\CC$-vector space by $T_\mfk$ for $\mfk \in A_+$ with $\deg \mfk \leq \deg \nfk-2$.
\end{cor}

\subsection{Isogeny between elliptic curves}\label{subsec: Ell}

Let $E$ be an elliptic curve over $K$ which has split multiplicative reduction at the place $\infty$.
For each prime $\pfk \in A_+$, let
\begin{eqnarray}\label{eqn: ap(E)}
&& a_\pfk(E) := \begin{cases}
|\pfk|_\infty + 1 - \#\overline{E}(\FF_\pfk) & \text{ if $E$ has good reduction at $\pfk$,} \\
1 & \text{ if $E$ has split multiplicative reduction at $\pfk$,} \\
-1 & \text{ if $E$ has non-split multiplicative reduction at $\pfk$,} \\
0 & \text{ if $E$ has additive reduction at $\pfk$.}
\end{cases}
\end{eqnarray}
Here $\FF_\pfk := A/\pfk$ and $\overline{E}$ denotes the reduction of $E$ at $\pfk$.
Let $\nfk\, \infty$ be the conductor of $E$ for some $\nfk \in A_+$.
From the work of Weil, Jacquet-Langlands, Grothendieck, Deligne, Drinfeld and Zarhin, there exists a unique $f_E \in \Hcal_0^{\text{new}}(\nfk)$ such that 
\begin{eqnarray}\label{eqn: Modularity}
\begin{cases}
    c_1(f_E) = 1; \\
    f_E|T_\mfk = c_\mfk(f_E)\, f_E \text{ for every $\mfk \in A_+$;} \\
    c_\pfk(f_E) = a_\pfk(E) \text{ for every prime $\pfk \in A_+$.}
\end{cases}
\end{eqnarray}
Moreover $f_E$ only depends on the $K$-isogeny class of $E$ (\cite{G-R}). Using our Sturm-type bound, we are able to determine effectively when two such given elliptic curves over $K$ are isogenous.



\begin{proof}[Proof of Corollary~\ref{cor: 0.6}]
We only have to prove the converse statement. Let $f_{E_1}$, $f_{E_2}$ be the harmonic cochains in $\Hcal_0^{\text{new}}(\nfk)$ corresponding to $E_1$, $E_2$ respectively, and such that $a_\pfk (E_1) = a_\pfk(E_2)$ for any prime $\pfk$ as in the statement of the corollary. Then \eqref{eqn: Modularity} shows that $c_\mfk(f_{E_1}) = c_\mfk(f_{E_2})$ for every $\mfk \in A_+$ with $\deg \mfk \leq \deg \nfk - 2$ if $\nfk$ is either a prime power or square-free or $\nfk = \pfk^2 \qfk$ for primes $\pfk,\qfk \in A_+$ with $\deg \qfk =1$, and $\deg \mfk \leq \deg \nfk - 2 + \ell(\nfk)$ otherwise.
By Corollary~\ref{cor: SBound.3}, Proposition \ref{prop: SBound.newspace}, Lemma \ref{lem: p2q} and Proposition \ref{prop: Sbound.1} according to the several cases, we get $f_{E_1} = f_{E_2}$, therefore $E_1$ and $E_2$ are isogenous over $K$.
\end{proof}

\section{Sturm-type bound for Drinfeld modular forms}

In this section, we study an analogous problem for Drinfeld modular forms. 

\subsection{Drinfeld modular forms}\label{sec: SBoundD}

Here we recall the definition of Drinfeld modular forms and the basic properties to be used. For further details we refer to \cite[V.3]{Gek3} and \cite[Section 2]{G-R}.

Let $\CC_\infty$ be the completion of a chosen algebraic closure of $K_\infty$. The \emph{Drinfeld half plane} is $\Omega := \CC_\infty - K_\infty$, which has a rigid analytic structure and is equipped with a left action of $\GL_2(K_\infty)$ via  fractional linear transformations.
Given non-negative integers $k$ and $m$ with $0\leq m \leq q-2$, for a rigid holomorphic function $f: \Omega \rightarrow \CC_\infty$ we set
$$\quad
\forall \gamma = \begin{pmatrix}a&b \\ c&d\end{pmatrix} \in \GL_2(K_\infty),\ \forall z \in \Omega,\quad (f\big|_{k,m} [\gamma])(z) := (\det \gamma)^m (cz+d)^{-k}f\left(\frac{az+b}{cz+d}\right).$$

\begin{defn}\label{defn: D-module}
Let $\nfk \in A_+$.
A \emph{Drinfeld modular form of weight $k$ and type $m$ for $\Gamma_0(\nfk)$} is a rigid holomorphic function $f: \Omega \rightarrow \CC_\infty$ satisfying:
\begin{enumerate}
    \item\label{defn: fmd1}
    for all $\gamma \in \Gamma_0(\nfk)$, $f\big|_{k,m}[\gamma] = f ;
    $
    \item\label{defn: fmd2} $f$ is holomorphic at all cusps of $\Gamma_0(\nfk)$.
\end{enumerate}
We denote by $M_{k,m}(\nfk)$ the $\CC_\infty$-vector space of Drinfeld modular forms of weight $k$ and type~$m$ for $\Gamma_0(\nfk)$.
\end{defn}

To state condition \eqref{defn: fmd2} more precisely, we recall the $t$-expansions of Drinfeld modular forms at the cusps of $\Gamma_0(\nfk)$ as follows.
Let $t$ be given by $$\forall z \in \Omega,\quad t(z) := \sum_{a \in A} \frac{1}{z-a},$$
which is a holomorphic function on $\Omega$ satisfying $t(z+a) = t(z)$ for every $a \in A$.
Then $t(z)$ is a uniformizer at the cusp infinity.
Given $f \in M_{k,m}(\nfk)$,
condition \eqref{defn: fmd1} implies that 
$f(z+a) = f(z)$ for every $a \in A$. Thus $f$ can be written for any $z\in\Omega$ with $|t(z)|_\infty$ small enough, as
$$f(z) = \sum_{n \in \ZZ}a_n(f) \ t^n(z),$$
where $\{a_n(f)\in \CC_\infty \mid n \in \ZZ\}$ is uniquely determined by $f$. To shorten notation, we will omit the condition that $|t(z)|_\infty$ is small enough in what follows.
In general, for every $\gamma \in \Gamma$, one has
$$\forall a \in A, \quad (f|_{k,m}[\gamma])(z+\nfk_{\gamma} a) = (f|_{k,m}[\gamma])(z), $$
where $\nfk_\gamma$ is the width of $\gamma$, introduced after Remark~\ref{rem: Quo-G}.
Thus we may write
$$\forall z \in \Omega,\quad (f|_{k,m}[\gamma])(z) = \sum_{n \in \ZZ}a_n^\gamma(f) \ t\left(\frac{z}{\nfk_\gamma}\right) .$$
Condition \ref{defn: fmd2} says that $a_n^\gamma(f) = 0$ for every $n <0$ and $\gamma \in \Gamma$.

\begin{rem}\label{rem: Coefficients}
In fact, condition \eqref{defn: fmd1} for the matrix $\gamma = \begin{pmatrix}\varepsilon & 0 \\ 0&\varepsilon\end{pmatrix}$ with $\varepsilon \in \FF_q^\times$ tells that $M_{k,m}(\nfk) = 0$ unless $k \equiv 2m \bmod q-1$.
Moreover, since $t(\varepsilon z) = \varepsilon^{-1} t(z)$ for any $\varepsilon \in \FF_q^\times$,
choosing $\gamma = \begin{pmatrix} \varepsilon & 0 \\ 0&1\end{pmatrix}$ in \eqref{defn: fmd1} gives that
$$ a_n(f) = 0 \quad \text{ unless $n = m + (q-1)j$ with $j \in \ZZ_{\geq 0}$.}$$
As a consequence, we put $b_j(f):= a_{m+(q-1)j}(f)$ for $j \in \ZZ_{\geq 0}$ so that the $t$-expansion of $f$ is
$$f(z) = \sum_{j=0}^{\infty} b_j(f)\ t^{m+(q-1)j}(z).$$
\end{rem}

For $\ell \geq 0$, we are interested in 
$$M_{k,m}^{(\ell)}(\nfk) := \big\{f \in M_{k,m}(\nfk)\ \big|\ \forall \gamma\in\Gamma, \forall n< \ell,\ a_n^{\gamma}(f) = 0 \big\}.$$
It is called the space of \emph{$\ell$-cuspidal Drinfeld modular forms of weight $k$ and type $m$ for $\Gamma_0(\nfk)$}.

\subsection{Sturm-type bound for Drinfeld modular forms}\label{sec: SBoundD2}
The following is obtained by an argument similar to proofs of the Sturm bound for classical modular forms.

\begin{thm}\label{thm: SBound.D}
Given $\nfk \in A_+$, let $\kappa(\nfk) := [\Gamma: \Gamma_0(\nfk)]$. Then $f \in M_{k,m}^{(\ell)}(\Gamma_0(\nfk))$ is identically zero if
$$b_j(f) = 0 \quad \text{ for every }\ 0\leq j \leq \kappa(\nfk) \cdot \left(\frac{k}{q^2-1}-\frac{\ell}{(q-1)|\nfk|_\infty}\right) + \frac{\ell - m|\nfk|_\infty}{(q-1)|\nfk|_\infty}.$$
\end{thm}

\begin{rem}
Using Proposition~4.3 of \cite{C} for $0 \leq m \leq q-2$, we have:
\[
\dim M_{k,m}(1) = 1+ \left\lfloor \frac{k}{q^2-1} - \frac{m}{q-1} \right\rfloor.
\]
Thus the bound of the theorem is sharp for $\nfk = 1$.
\end{rem}

\begin{proof}
Suppose that $\nfk =1$. For $f \in M_{k,m}(1)$, let $\ord_t(f)$ denote the order of vanishing of $f$ with respect to the uniformizer $t$. Suppose that $f \in  M_{k,m}^{(\ell)}(1)$ with $b_j(f)=0$ for every $j \leq \frac{k}{q^2-1} - \frac{m}{q-1}$. Then we have
\[
f(z) = \sum_{j > \frac{k-(q+1)m}{q^2-1}} b_j(f)\ t^{m+(q-1)j}(z)
\]
hence $\ord_t(f)> \frac{k}{q+1}$. But according to Gekeler's valence formula for Drinfeld modular forms for $\Gamma$ (\cite[(5.14)]{Gek5}), any nonzero $g \in M_{k,m}(1)$ satisfies
\[
\ord_t(g) \leq \frac{k}{q+1}.
\]
Consequently, $f$ is identically zero.


\smallskip

Suppose that $\deg \nfk > 1$. Let $\kappa = \kappa(\nfk)$ and $\left(\begin{smallmatrix}1&0\\0&1\end{smallmatrix}\right)= \gamma_1 ,\gamma_2,...,\gamma_\kappa$ be representatives of the right cosets of $\Gamma_0(\nfk)$ in $\Gamma$.
Given $f \in M_{k,m}^{(\ell)}(\nfk)$,
put 
$$\tilde{f} := \prod_{i=1}^\kappa f\big|_{k,m}[\gamma_i].$$
Then $\tilde{f}$ is a Drinfeld modular form of weight $\kappa \cdot k$ and type $\tilde{m}$ for $\Gamma$, where
$0\leq \tilde{m} \leq q-2$ is such that $\tilde{m} \equiv \kappa \cdot m \bmod (q-1)$.
Considering the $t$-expansions of $\tilde{f}$ and $(f_{\mid k,m}[\gamma_i])_{2 \leq i \leq \kappa}$ we have
\begin{eqnarray}\label{eqn: expansion}
\tilde{f}(z) &=& \sum_{j=0}^\infty b_j(\tilde{f})\ t^{\tilde{m}+(q-1)j}(z) \nonumber \\ 
&=& \left(\sum_{j = 0}^\infty b_j(f)\ t^{m+(q-1)j}(z)\right) \cdot  \prod_{i=2}^\kappa \left(\sum_{n=0}^\infty a_n^{\gamma_i}(f)\  t^n\left((\nfk/\nfk_{\gamma_i})\frac{z}{\nfk}\right)\right).
\end{eqnarray}
Note that for $0 \neq a \in A$ and $z \in \Omega$ with $|t(z)|_\infty$ small enough, one has
$$t(az) = \sum_{i=0}^\infty c_i \; t^{|a|_\infty + i}(z), \quad \text{ where $c_i \in \CC_\infty$ for $i \geq 0$.}$$
Since $f$ is $\ell$-cuspidal, one has $a_n^{\gamma_i}(f) = 0$ for $n \leq \ell$ and $2 \leq i \leq \kappa$. 
Now suppose 
$$b_j(f) = 0 \quad \text{ for } \ 
j \leq \kappa \cdot \left(\frac{k}{q^2-1}-\frac{\ell}{(q-1)|\nfk|_\infty}\right) + \frac{\ell - m|\nfk|_\infty}{(q-1)|\nfk|_\infty}.$$
Let $t_\nfk(z) := t(z/\nfk)$ for $z \in \Omega$.
Expressing as $t_\nfk$-expansions on both sides of \eqref{eqn: expansion},
we then obtain that $b_j(\tilde{f}) = 0$ for $j \leq k\kappa/(q^2-1) - \tilde{m}/(q-1)$. From the case $\nfk =1$ proved previously, we get $\tilde{f}$ is identically zero. Since the ring of rigid analytic functions on $\Omega$ is an integral domain, $f$ is identically zero.
\end{proof}

\begin{rem}
For each prime $\pfk \in A_+$, the Hecke operator $T_\pfk$ on $M_{k,m}(\nfk)$ is given by (following \cite[Section 4.3]{Arm1}):
$$\forall z \in \Omega,\quad (f|_k T_\pfk)(z) :=  \pfk^{-1} \sum_{\subfrac{u \in A}{\deg u < \deg \pfk}}f\left(\frac{z+u}{\pfk}\right) + \mu_\nfk(\pfk)\cdot \pfk^{k-1} f(\pfk z).$$
Here $\mu_\nfk(\pfk) = 1$ if $\pfk \nmid \nfk$ and $0$ otherwise. 
In general, for $\mfk \in A_+$ written as $\mfk = \pfk_1^{r_1} \cdots \pfk_t^{r_t}$ with distinct primes $\pfk_1,\ldots,\pfk_t$, put
$$T_\mfk := \prod_{i=1}^t T_{\pfk_i}^{r_i}  \quad\in \End_{\CC_\infty}\big(M_{k,m}^{(\ell)}(\nfk)\big).$$
Let $\mathbf{T}^{(\ell)}_{k,m}(\nfk)$ be the $\CC_\infty$-algebra generated by $(T_\mfk)_{\mfk \in A_+}$. The first coefficient $b_1$ provides the following pairing:
$$
\begin{tabular}{ccccc}
$M_{k,m}^{(\ell)}(\nfk)$ & $\times$ & $\mathbf{T}_{k,m}^{(\ell)}(\nfk)$ & $\longrightarrow$ & $\CC_\infty$ \\
(\ $f$ &,& $T_\mfk$\ ) & $\longmapsto$ & $b_1(f|_k T_\mfk)$.
\end{tabular}
$$
However, unlike the classical case, this pairing is not expected to be perfect in general, cf.\ \cite[Theorem 1.1 and Conjecture 6.9]{Arm1}.
Besides, the action of Hecke operators on the $t$-expansion of $f$ is not well-understood.
Therefore our Sturm-type bound for Drinfeld modular forms does not directly provide a finite family of Hecke operators generating the $\CC_\infty$-algebra  $\mathbf{T}_{k,m}^{(\ell)}(\nfk)$.
\end{rem}

\subsubsection*{Acknowledgements}
The first author was partially supported by the French ANR program through project FLAIR (ANR-17-CE40-0012) and the French "Investissements d'Avenir" program through project ISITE-BFC (contract ANR-lS-IDEX-OOOB). She warmly thanks the National Center for Theoretical Sciences in Taipei and Hsinchu, where she carried out this work while visiting, for the financial support and excellent working conditions.

The second author is supported by the National Center for Theoretical Sciences and the Ministry of Science and Technology (grant no.\ 107-2628-M-007-004-MY4).

Both authors would like to thank the organizers of the stimulating conference \lq\lq New developments in the theory of modular forms over function fields \rq\rq\ held in Pisa, 2018.

\end{document}